\documentclass[12pt]{article}
\usepackage{amsthm}
\usepackage{amsmath,amssymb,euscript,mathrsfs,bm,pstricks}

\usepackage{eepic,bbm}

\usepackage{bbold}
\usepackage{dsfont}

\usepackage{ebgaramond}

\usepackage{setspace}
\usepackage[shortlabels]{enumitem}

\usepackage{tikz-cd}

\usepackage[colorlinks=true,linkcolor=black,citecolor=black,urlcolor=black]{hyperref}

\pushQED{\qed}

\psset{unit=1pt}
\psset{arrowsize=4pt 1}
\psset{linewidth=.5pt}

\countdef\x=23
\countdef\y=24
\countdef\z=25
\countdef\t=26

\def\graybox(#1,#2){
\x=#1 \y=#2 
\z=\x \t=\y
\advance\z by 10 
\advance\t by 10 
\psframe[fillstyle=solid,fillcolor=lightgray,linewidth=0pt](\x,\y)(\z,\t) 
\psline[linewidth=.5pt](\x,\y)(\x,\t)(\z,\t)(\z,\y)(\x,\y)}

\newcommand{\define}{\textbf}

\newcommand{\excise}[1]{}

\newcommand{\isom}{\cong}

\renewcommand{\phi}{\varphi}

\renewcommand{\tilde}{\widetilde}
\renewcommand{\hat}{\widehat}
\renewcommand{\bar}{\overline}

\newcommand{\C}{\mathds{C}}

\newcommand{\Z}{\mathds{Z}}
\renewcommand{\P}{\mathds{P}}

\newcommand{\bV}{\mathds{V}}

\DeclareMathOperator{\rk}{rk}

\DeclareMathOperator{\sh}{sh}

\DeclareMathOperator{\Bij}{Bij}
\DeclareMathOperator{\Inj}{Inj}
\newcommand{\Sgp}{\mathcal{S}}        

\newcommand{\tto}{\twoheadrightarrow}

\newcommand{\bull}{ {\scriptscriptstyle{\bullet}}  }

\newcommand{\bbe}{\mathbb{e}}     
\newcommand{\bbc}{\mathbb{c}}

\newcommand{\Gr}{ {Gr} }
\newcommand{\Fl}{ {Fl} }

\newcommand{\sGr}{\mathrm{Gr}}    
\newcommand{\sFl}{\mathrm{Fl}}    

\newcommand{\LG}{LG}
\newcommand{\sLG}{\mathrm{LG}}    

\newcommand{\IG}{IG}
\newcommand{\sIG}{\mathrm{IG}}    

\newcommand{\sIFl}{\mathrm{Fl}^C}

\newcommand{\Tz}{\mathbf{T}}       

\newcommand{\bp}{\bm{p}}

\newcommand{\triple}{{\bm\tau}}

\newcommand{\bGamma}{\mathbf{\Gamma}}

\newcommand{\sS}{\mathbf{S}}              
\newcommand{\enS}{\mathbf{S}}              
\newcommand{\tenS}{\bm{S}}              
\newcommand{\SSS}{\mathfrak{S}}        
\newcommand{\bsS}{\overleftarrow{\mathfrak{S}}}  

\newcommand{\bbOmega}{\mathbb{\Omega}}

\newcommand{\tS}{S}             
\newcommand{\tbS}{\mathds{S}}   

\newtheorem{theorem}{Theorem}[section]

\newtheorem{proposition}[theorem]{Proposition}
\newtheorem{corollary}[theorem]{Corollary}

\theoremstyle{definition}
\newtheorem{definition}[theorem]{Definition}
\newtheorem{remark}[theorem]{Remark}
\newtheorem{example}[theorem]{Example}

\begin{document}

\title{Infinite flags and Schubert polynomials}
\author{David Anderson\thanks{Partially supported by NSF CAREER DMS-1945212.}}
\date{May 31, 2023}

\maketitle

\renewcommand{\bfseries}{\upshape}

\begin{abstract}
We study Schubert polynomials using geometry of infinite-dimensional flag varieties and degeneracy loci.  Applications include Graham-positivity of coefficients appearing in equivariant coproduct formulas and expansions of back-stable and enriched Schubert polynomials.  We also construct an embedding of the type C flag variety, and study the corresponding pullback map on (equivariant) cohomology rings.
\end{abstract}

\setcounter{tocdepth}{1}

\tableofcontents

\renewcommand{\bfseries}{\itshape}

\section{Introduction}\label{s.intro}

Schubert polynomials represent the classes of Schubert varieties in the cohomology ring of a flag variety.  For $\Fl(\C^n)$, Schubert varieties $\Omega_w$ are indexed by permutations $w\in \Sgp_n$, and their classes form an additive basis of the cohomology ring.  The ring $H_T^*\Fl(\C^n)$ has a Borel presentation as $\Z[x_1,\ldots,x_n,y_1,\ldots,y_n]/I$, so some choices are involved in lifting a class to a polynomial.  Among these choices, the polynomials $\SSS_w(x;-y)$, introduced by Lascoux and Sch\"utzenberger in 1982, are widely accepted as the nicest representatives for $[\Omega_w]$, because of their many wonderful combinatorial, algebraic, and geometric properties \cite{LS}.

One of these properties is {\em stability} with respect to embeddings of flag varieties: the same polynomial represents $\Omega_w$, whether one considers the permutation $w$ in $\Sgp_n$, or in $\Sgp_{n+1}$, or in any $\Sgp_m$ for $m\geq n$.  As part of a search for analogous Schubert polynomials for flag varieties of other types, Fomin and Kirillov enumerated a list of desirable properties possessed by $\SSS_w$, including a version of stability among them \cite{FK}.  Around the same time, Billey and Haiman used stability (of a subtly different sense from that of \cite{FK}) as a defining property for Schubert polynomials in classical types \cite{BH}.

The operative fact used by Billey and Haiman is this: in the limit, the relations defining cohomology rings disappear, and one obtains canonical polynomials representing Schubert classes.  In type C, one builds an infinite isotropic flag variety starting with a union of Lagrangian Grassmannians.  The Billey-Haiman polynomials are, by definition, stable Schubert classes in the limiting cohomology ring, which is a polynomial ring over a nontrivial base ring $\Gamma$.  
The analogous construction in type A leads not to the Lascoux-Sch\"utzenberger polynomials, but rather to the {\em enriched Schubert polynomials} to be studied here.  (A more precise description of the analogy is at the end of this introduction.)  These polynomials, denoted $\enS_w(c;x;y)$, have coefficients in a nontrivial base ring $\Lambda$, and they specialize to $\SSS_w(x;{-y})$ under a canonical quotient $\Lambda \to \Z$.  The same holds also for the (essentially equivalent) {\em back-stable Schubert polynomials} recently studied by Lam, Lee, and Shimozono, building on ideas of Buch and Knutson, although there the perspective is reversed, the correspondence with Schubert classes being a theorem rather than a definition \cite[\S6]{LLS}.

The subject of this article is a variation on \cite{LLS} and \cite{part1}.  Using the geometry of certain infinite-dimensional flag varieties, we provide an alternative construction of the back-stable Schubert polynomials---in the guise of enriched Schubert polynomials \cite{part1}.  These constructions lead naturally to alternative proofs of basic properties of these polynomials, and we include some of these arguments.

When discussing infinite-dimensional flag varieties, some care must be taken to distinguish among several constructions.  The main players in our story will be the {\em Sato flag variety} and {\em Sato Grassmannian}.  All the other flag varieties embed in these, including varieties parametrizing finite-dimensional (or finite-codimensional) subspaces, and infinite isotropic (type C) flag varieties.  The affine flag varieties and Grassmannians also embed, and these embeddings are used in \cite{eq-aff}.

All our infinite-dimensional flag varieties are limits of finite-dimensional ones, so they may be regarded as devices for keeping track of stability: one can always translate statements about infinite-dimensional varieties into statements about compatible sequences of finite-dim{\-}ensional varieties.  This is sometimes worked out explicitly, and sometimes left implicit; given the statements, there is generally little trouble in supplying proofs.

Some new features are more salient in the infinite setting, though.  Here we focus on morphisms among various Grassmannians and flag varieties, and their effect on Schubert polynomials.  The {\em direct sum} morphisms are particularly interesting: we use them to study a coproduct on equivariant cohomology (\S\ref{s.dirsum}).  For instance, the coproduct of a Schubert class $[\Omega_\lambda]$ in the Sato Grassmannian is
\[
  [\Omega_\lambda] \mapsto \sum_{\mu,\nu} \hat{c}_{\mu,\nu}^\lambda(y) [\Omega_\mu] \otimes [\Omega_\nu], 
\]
for some polynomials $\hat{c}_{\mu,\nu}^\lambda(y)$, called {\em dual Littlewood-Richardson polynomials}  \cite{molev}.  Computing the coproduct via the direct sum morphism, we give a direct proof that these polynomials (and variations of them) satisfy Graham-positivity (Theorems~\ref{t.positive}, \ref{t.positive2}, and \ref{t.positiveC}).  The first of these positivity results was proved in \cite{LLS} by passing through the quantum-affine correspondence.  The second involves two sets of equivariant parameters $y$ and $y'$, and was suggested in \cite{LLS}, but not proved.  The third is an analogue in type C, and appears to be new.

The direct sum morphism also leads to a way of computing the equivariant coproduct coefficients $\hat{c}_{\mu,\nu}^\lambda(y)$, by expanding a product of one double Schur polynomial by a double Schur with permuted $y$-variables; a similar method computes the flag variety variants (Proposition~\ref{c.coprod-schur-schub}). 
While the idea of using direct sum in relation to coproduct has many antecedents (e.g., \cite{bergeron-sottile,buch,thomas-yong,knutson-lederer,LLS}), I do not know of instances where it has been used in the equivariant setting.

Much of this article has close parallels in \cite{LLS}.  Two technical points of contrast are worth highlighting.  First, as will be made clear in the constructions of \S\ref{s.sato}, the Sato flag variety $\sFl$ considered here is larger than that of \cite{LLS}; this has the effect of making the equality $H_T^*\sFl = \Lambda[x;y]$ a calculation rather than a convention, and it also allows the affine flag variety to embed in $\sFl$.  Second, and perhaps more substantially, we do not insist on a ``GKM''-type description of equivariant cohomology, although we do include a discussion of fixed points.  Instead, cohomology rings are presented in terms of Chern class generators.  This allows us to use smaller torus actions, with larger fixed loci, which are needed in the construction of the direct sum morphisms.

The re-interpretation of back-stable Schubert polynomials was not the original motivation for this work; the connection became apparent (to me) only after the fact.  The constructions were forced by requiring that the stability one sees in the type C polynomials of Billey-Haiman should be compatible with natural embeddings of the symplectic Grassmannians and flag varieties inside the usual (type A) ones.  This basic notion guides much of what we do here.  As a preview, let us index a basis for $\C^{2n}$ as $e_{-n+1},\ldots,e_{0},e_1,\ldots,e_n$,  
and define a symplectic form so that
\[
  \langle e_{1-i}, e_i \rangle = -\langle e_i, e_{1-i} \rangle = 1
\]
for $i>0$, and all other pairings are $0$.  The inclusions
\[
 \C^{2n} \hookrightarrow \C^{2n+2} = \C\cdot e_{-n} \oplus \C^{2n} \oplus \C\cdot e_{n+1}
\]
lead to embeddings of Lagrangian Grassmannians $\LG(n,\C^{2n}) \hookrightarrow \LG(n+1,\C^{2n+2})$, defined by $E \mapsto \C\cdot e_{-n} \oplus E$.  
The same maps define embeddings of ordinary Grassmannians, 
so that the diagram
\[
\begin{tikzcd}
 \LG(n,\C^{2n}) \ar[r,hook] \ar[d,hook] & \LG(n+1,\C^{2n+2}) \ar[d,hook] \\
 \Gr(n,\C^{2n}) \ar[r,hook] & \Gr(n+1,\C^{2n+2})
\end{tikzcd}
\]
commutes.  Taking appropriate limits of cohomology rings, for the type A Grassmannian one sees the ring of symmetric functions $\Lambda$, and for the Lagrangian Grassmannian, the ring $\Gamma$ of $Q$-functions.  In the limit, pullback by the embedding $\LG(\C^{2n})\subset \Gr(n,\C^{2n})$ corresponds to a canonical surjection $\Lambda \tto \Gamma$.  (In symmetric function theory, one often sees an inclusion $\Gamma \hookrightarrow \Lambda$; this also arises from a morphism between infinite Grassmannians, but a less natural one from our perspective.  See Remark~\ref{r.gamma-lambda}.)

Similar maps define embeddings of flag varieties.  The system of embeddings for symplectic (type C) varieties is what Billey and Haiman use to define type C Schubert polynomials.  The limit of the compatible embeddings in type A leads directly to the Sato flag variety, and to enriched Schubert polynomials $\sS_w(c;x;y)$ corresponding to Schubert classes.  When one evaluates the $c$ variables as certain symmetric functions (in an infinite variable set), these polynomials become the back-stable Schubert polynomials of \cite{LLS}.

Many basic properties of these polynomials were enumerated in \cite{part1}, inspired by similar properties of the back-stable polynomials \cite{LLS}.  In summary, the overall aim of this article is to examine those aspects of Schubert polynomials for which the geometry of infinite flag varieties provides a new or useful perspective---particularly, what happens to Schubert classes under various morphisms of flag varieties.

\bigskip
\noindent
{\it Acknowledgements.}  This work grew out of a joint project with William Fulton, and I thank him for our long-running collaboration, for encouraging me to pursue this extension, and for vital feedback along the way.  
My great debt to the authors of \cite{LLS} should be evident.  Much of what I have learned about infinite-dimensional flag varieties began with lectures and papers by Mark Shimozono, and I would like to thank him in particular for his lucid and down-to-earth exposition.  I am grateful to Allen Knutson for clarifying conversations about \cite{knutson-lederer} and about fixed points.

\section{Preliminaries}

%
\subsection{Permutations}

With some modifications, we follow \cite{LLS} for permutations.

We write $\Bij( X )$ for the group of all bijections of a set $X$ to itself.  We will only consider subsets $X\subseteq \Z$, and we focus on the subgroup $\Sgp_\Z \subseteq \Bij(\Z)$ consisting of all $w$ such that $\{ i\in \Z\,|\, w(i)\neq i \}$ is finite---that is, $w$ fixes all but finitely many integers.  Some variations will be discussed in \S\ref{s.fp}.

The subgroup $\Sgp_{\neq0}$ is $\Sgp_+\times \Sgp_-$, where $\Sgp_+ = \Sgp_\Z \cap \Bij(\Z_{>0})$ and $\Sgp_- = \Sgp_\Z \cap \Bij(\Z_{\leq0})$.  That is, $\Sgp_{\neq 0}$ is the subgroup of $\Sgp_\Z$ preserving the subsets of positive and non-positive integers.

For finite intervals $[m,n]$, we usually write $\Sgp_{[m,n]}  = \Bij([m,n])$, and $\Sgp_n = \Sgp_{[1,n]}$ for $n>0$.  We have
\[
  \Sgp_+ = \bigcup_{n>0} \Sgp_{[1,n]}, \quad \Sgp_- = \bigcup_{n>0} \Sgp_{[-n,0]}, \quad \text{and} \quad \Sgp_\Z = \bigcup_{n>0} \Sgp_{[-n,n]}.
\]

Elements $w\in\Sgp_\Z$ are written in one-line notation: choose an interval $[m,n]$ so that $w(i)=i$ for all $i$ outside $[m,n]$, and write $w=[w(m),\ldots,w(n)]$.

\define{Bruhat order} on $\Sgp_\Z$ is defined as follows.  For each $p,q\in \Z$ and $w\in \Sgp_\Z$, we set
\[
  k_w(p,q) = \# \{ a\leq p \,|\, w(a) > q \}.
\]
Then $v\leq w$ in $\Sgp_\Z$ if $k_v(p,q)\leq k_w(p,q)$ for all $p,q\in \Z$.

An element $w\in \Sgp_\Z$ is \define{Grassmannian} if it has no descents except possibly at $0$, so $w(i)<w(i+1)$ for all $i\neq 0$.  Grassmannian elements are in correspondence with partitions $\lambda$: given a Grassmannian permutation $w$, the partition $\lambda = (\lambda_1\geq \lambda_2 \geq \cdots \geq 0)$ is defined by $\lambda_k = w(1-k) -1+k$, for $k>0$.  Conversely, given $\lambda$, one defines $w=w_\lambda$ by setting $w_\lambda(k)=\lambda_{1-k}+k$ for $k\leq 0$, and then filling in the positive values with the unused integers in increasing order.

The \define{length} $\ell(w)$ of $w\in \Sgp_\Z$ is the cardinality of the (finite) set $\{i<j\,|\, w(i)>w(j)\}$.  

The element $w_\circ^\infty \in \Bij(\Z)$ defined by $w_\circ^\infty(i) = 1-i$ does not lie in $\Sgp_\Z$, but conjugation by $w_\circ^\infty$ defines a length-preserving outer automorphism $\omega$ of $\Sgp_\Z$:
\[
 \omega(w)(i) = (w_\circ^\infty w w_\circ^\infty)(i) = 1-w(1-i).
\]

\subsection{Vector spaces}

Let $V$ be a countable-dimensional vector space with basis $e_i$ for $i\in \Z$.  For any interval $[m,n]$, there is a subspace $V_{[m,n]}$ with basis $e_i$ for $i\in [m,n]$.  
For semi-infinite intervals we usually write $V_{\leq n}$, or $V_{>m}$.  
The \emph{standard flag} $V_{\leq\bullet}$ in $V$ has components $V_{\leq k}$ with basis $e_i$ for $i\leq k$, for each $k\in \Z$.  The \emph{opposite flag} $V_{>\bullet}$ is comprised of spaces ${V}_{>k}$ spanned by $e_i$ for $i>k$.  Clearly $V=V_{\leq0} \oplus V_{>0}$ (and $V=V_{\leq k} \oplus V_{>k}$ for any $k$).

When the context is clear, we use the same notation for standard and opposite flags in $V_{(m,m]}$, for instance writing $V_{\leq k} \subseteq V_{(m,n]}$ instead of $V_{(m,k]} \subseteq V_{(m,n]}$.

A torus $T$ acts on $V$, so that $e_i$ is scaled by the character $y_i$, for $i\in \Z$.  So $T$ also acts on each subspace $V_{[m,n]}$.  We generally take $T$ to be the countable product $T=\prod_{i\in\Z} \C^*$, so that its classifying space is $\prod_{i\in\Z} \P^\infty$.  This is an inverse limit of finite products of $\P^\infty$, so the $T$-equivariant cohomology of a point is a polynomial ring in the $y$ variables:
\[
  H_T^*(\mathrm{pt}) = \Z[y] = \Z[\ldots,y_{-1},y_0,y_1,\ldots].
\]
(For those who prefer finite dimensional groups, one may also take $T$ to be any torus, with weights $y_i$, for $i\in\Z$.  By taking $T$ sufficiently large, any given finite set of $y$'s can be made algebraically independent.)

\subsection{Flag varieties}

For any vector space $W$, the flag variety $\Fl_+(W)$ is the space of all complete flags of finite-dimensional subspaces of $W$.  That is, a point of $\Fl_+(W)$ is $E_\bullet = (0 \subset E_1 \subset E_2 \subset \cdots \subset W)$, where $\dim E_i = i$. 
When $W$ is finite-dimensional, this is the usual complete flag variety.  In general, it is a limit of finite-dimensional flag varieties: to construct $\Fl_+(W)$, for each $d>0$, one forms $\Gr(d,W)$ as the union of $\Gr(d,U)$ over finite-dimensional subspaces $U\subset W$; then $\Fl_+(W)$ embeds naturally in the product $\prod_{d>0} \Gr(d,W)$.  So $\Fl_+(W)$ inherits its topology from the product topology on the Grassmannians.  This is the same as the inverse limit topology with respect to projections onto partial flag varieties.

There is also a variety $\Fl_-(W)$ parametrizing flags of finite-codimensional subspaces of $W$, but here an extra requirement is imposed: one fixes a flag $W^\bullet$ of finite-codimen{\-}sional subspaces of $W$.  Then a point of $\Fl_-(W)$ is $E^\bullet = ( \cdots \subset E^2 \subset E^1 \subset W)$, where $E^i$ has codimension $i$ in $W$, and each $E^i$ contains some $W^j$.  (Often we negate indices and write $E_{-i}=E^i$ for such flags.)  Equivalently, let $K_i = W/W^i$, and consider the {\it restricted dual space} $W^{*'} = \bigcup_i K_i^*$.  (This is finite-dimensional when $W$ is, and countable-dimensional if $\dim W$ is infinite.)  Then $\Fl_-(W) = \Fl_+(W^{*'})$.

In our setting, an equivalent construction of these varieties is as follows.  The  flag variety $\Fl(1,\ldots,n;V_{>0})$ is a union of finite-dimensional partial flag varieties $\Fl(1,\ldots,n;V_{[1,m]})$ over $m\geq n$, with respect to standard embeddings coming from $V_{[1,m]} \subset V_{[1,m+1]}$.

The finite-dimensional flag varieties have tautological bundles $\tS_i$, and $T$ acts, restricting its action on $V$.  Taking the graded inverse limit of cohomology rings, one has
\[
  H_T^*\Fl(1,\ldots,n;V_{>0}) = \Z[y][x_1,\ldots,x_n],
\]
where $x_i$ restricts to $-c_1^T(\tS_i/\tS_{i-1})$ on each finite-dimensional variety.

Next we take the inverse limit of $\Fl(1,\ldots,n;V_{>0})$ over $n$, using natural projections.  (So it is a ``pro-ind-variety'': the inverse limit of a direct limit of algebraic varieties.)  Its equivariant cohomology is the direct limit of rings $\Z[y][x_1,\ldots,x_n]$ as $n\to\infty$, so
\[
  H_T^*\Fl_+(V_{>0}) = \Z[y][x_1,x_2,\ldots].
\]

Similarly, the construction of $\Fl_-(V_{\leq0})$ (with respect to the standard flag $V_{\leq\bullet}$) realizes it as a limit of the flag varieties $\Fl(m-n,\ldots,m;V_{(-m,0]})$, which have tautological bundles $\tS_i$ of codimension $-i$, for $i\leq 0$.  Its equivariant cohomology is
\[
   H_T^*\Fl_-(V_{\leq0}) = \Z[y][x_0,x_{-1},\ldots],
\]
where again $x_i$ restricts to $-c_1^T(\tS_{i}/\tS_{i+1})$ on each finite-dimensional variety, for $i\leq 0$.

\begin{remark}
One sometimes sees yet another limit, taking a union $\bigcup_{n>0}\Fl(V_{[1,n]})$ over the standard embeddings $V_{[1,n]} \subset V_{[1,n+1]}$.  This leads to what might be called a {\it restricted flag variety} $\Fl'_+(V_{>0})$, parametrizing flags $E_\bullet$ of finite-dimensional subspaces which are eventually standard: $E_k = V_{\leq k}$ for all $k \gg0$.  As a direct limit, its cohomology is
\[
 H_T^*\Fl'_+(V_{>0}) = \Z[y][\![x]\!]_{\mathrm{gr}},
\]
the ring of graded power series in $x$ with coefficients in $y$.  (For example, the infinite sum $\sum_{i>0} x_i$ is an element of this ring.)  The embedding $\Fl'_+(V_{>0}) \hookrightarrow \Fl_+(V_{>0})$ corresponds to the inclusion of the polynomial ring $\Z[y][x] \hookrightarrow \Z[y][\![x]\!]_{\mathrm{gr}}$.

We will not make use of these restricted varieties, except to mention their appearance in the literature.  One of several advantages of working with $\Fl_+(V_{>0})$ rather than $\Fl'_+(V_{>0})$ is that elements of its cohomology are automatically polynomials.
\end{remark}

\subsection{A technical note on limits}

For a rising union of spaces $X= \bigcup X_n$, the direct limit topology is defined so that a subset $U\subset X$ is open exactly when each intersection $U\cap X_n$ is open.  For an inverse system of spaces $\cdots \to X_n \to X_{n-1} \to \cdots$, the inverse limit topology on $X= \varprojlim X_n$ is the coarsest topology so that the projections $X \to X_n$ are continuous; in our context this is a subspace of the product topology on $\prod X_n$.

From the contravariance of cohomology, one may naively expect that
\[
  H^*\left( \bigcup X_n \right) = \varprojlim H^*(X_n) \quad \text{ and } \quad H^*\left( \varprojlim X_n \right) = \varinjlim H^*(X_n).
\]
Using \v{C}ech-Alexander-Spanier cohomology, and for the relatively nice topological spaces we encounter, these naive expectations hold.  For finite-dimensional algebraic varieties, this cohomology theory agrees with the more familiar singular cohomology.  These facts may be gleaned from standard algebraic topology texts; see also \cite[Appendix~A]{ecag}.

\section{Sato Grassmannians and flag varieties}\label{s.sato}

The primary focus of this article is on a different type of infinite-dimensional flag variety.  The {\it Sato Grassmannian} parametrizes subspaces of $V$ which are infinite in both dimension and codimension (but satisfy some other requirements).  It can also be described as a certain union of finite-dimensional Grassmannians.  The {\it Sato flag variety} similarly parametrizes flags of spaces belonging to Sato Grassmannians.  The constructions presented in this section are variations on ones found in \cite{LLS}, which in turn are based on Kashiwara's construction of thick flag manifolds \cite{kashiwara}, as well as certain Hilbert manifolds used as models for loop groups \cite{pressley-segal}.

Fixing our base flag $V_{\leq\bullet}$ as before, and an integer $k$, the \define{Sato Grassmannian} $\sGr^k$ is the set of all subspaces $E \subseteq V$ such that
\begin{enumerate}[(1)]
\item $V_{\leq -m} \subseteq E \subseteq V_{\leq m}$ for some $m>0$ (and hence all $m\gg0$), and \label{sG.cond1}

\item $\dim E/(E\cap V_{\leq 0}) - \dim V_{\leq 0}/(E\cap V_{\leq 0}) = k$. \label{sG.cond2}
\end{enumerate}
The first condition implies that both $E/(E\cap V_{\leq 0})$ and $V_{\leq 0}/(E\cap V_{\leq 0})$ are finite-dimensional, so the second condition makes sense.

This space depends on the base flag, and occasionally it is useful to indicate this dependence in the notation, writing $\sGr^k(V;V_{\leq\bullet})$.  On the other hand, we use the case $k=0$ frequently, so we sometimes drop the superscript and write $\sGr = \sGr^0$.

Condition \ref{sG.cond1} means  that $E\subset V$ comes from a point in $\Gr(m+k,V_{(-m,m]})$ for some $m$ and $k$, by mapping $E_{m+k} \subseteq V_{(-m,m]}$ to $V_{\leq -m} \oplus E_{m+k} \subseteq V_{\leq -m}\oplus V_{(-m,m]} = V_{\leq m}$.  Condition \ref{sG.cond2} specifies $k$.

Using this observation, for $k=0$ one constructs (and topologizes) the Sato Grassmannian $\sGr=\sGr^0$ as the union
\[
\sGr = \bigcup_{m\geq 0} \Gr(m,V_{(-m,m]})
\]
of finite-dimensional Grassmannians, using the embeddings $\Gr(m,V_{(-m,m]}) \hookrightarrow \Gr(m+1,V_{(-m-1,m+1]})$ which map an $m$-dimensional subspace $E_m$ of $V_{(-m,m]}$ to the $(m+1)$-dimensional subspace $\C\cdot e_{-m} \oplus E$ of $V_{(-m-1,m+1]}$.

Similarly, for any $k\in \Z$ one has
\[
\sGr^k = \bigcup_{m\geq |k|} \Gr(m+k,V_{(-m,m]}).
\]
(Without changing the result, these limits could be refined to run over $\Gr(m+k;V_{(-m,m']})$, for $m,m'\geq 0$, since these are co-final with $\Gr( m+k, V_{(-m,m]} )$.)

These unions are compatible with actions of $T$, so $T$ acts on $\sGr$.  Since $\sGr$ is a direct limit of finite-dimensional Grassmannians, the cohomology ring $H_{T}^*\sGr$ is the (graded) inverse limit:
\[
  H_{T}^*\sGr = \varprojlim_{m} H_{T}^*\Gr(m,V_{(-m,m]}) = \Z[y][c_1,c_2,\ldots] = \Lambda[y].
\]
Here $\Lambda = \Z[c_1,c_2,\ldots]$ is a polynomial ring;
the variable $c_i$ restricts to $c^{T}_i(V_{\leq0} - \tS_0)$ on each $\Gr(m,V_{(-m,m]})$, where $\tS_0 \subseteq V_{(-m,m]}$ is the tautological bundle of rank $m$.  From now on, we simply write $c_i = c^T(V_{\leq0} - S_0)$, with the notation $S_0$ standing for a tautological bundle on some large enough Grassmannian.

A similar calculation produces the same result for $H_{T}^*\sGr^k$, with variables $c_i^{(k)} = c_i^{T}(V_{\leq k} - \tS_k)$, so on each $\Gr(m+k,V_{(-m,m]})$, $\tS_k \subseteq V_{(-m,m]}$ is the tautological bundle of rank $m+k$.

The \define{Sato flag variety} is
\[
 \sFl  = \left\{ E_\bullet =( \cdots \subset E_{-1} \subset E_0 \subset E_1 \subset \cdots ) \,|\, E_k \in \sGr^k \right\},
\]
so it is a subvariety of $\prod_{k\in\Z} \sGr^k$.  Using the natural projections to $\prod_{|k|\leq n} \sGr^k$, it can be written as an inverse limit of a union of finite-dimensional partial flag varieties:
\[
 \sFl = \varprojlim_{n} \bigcup_{m} \Fl( m-n,\ldots,m,\ldots,m+n; V_{(-m,m]} ).
\]
Each such partial flag variety has a tautological flag of subbundles,
\[
 \tS_{-n} \subset \cdots \subset \tS_0 \subset \cdots \subset \tS_n \subseteq V_{(-m,m]},
\]
with $\tS_i$ of rank $m+i$.  (As with the Grassmannians, the limit can be taken over partial flag varieties $\Fl(m-n,\ldots,m'+n';V_{(-m,m']})$.)

The cohomology ring of the limit is computed as
\begin{align*}
 H_{T}^*\sFl &= \varinjlim_n \varprojlim_{m} H_{T}^* \Fl( m-n,\ldots,m,\ldots,m+n; V_{(-m,m]} ) \\
 & = \Lambda[y][\ldots,x_{-1},x_0,x_1,\ldots] = \Lambda[x;y],
\end{align*}
where $x_i = -c_1^{T}(\tS_i/\tS_{i-1})$ and $c_i = c_i^{T}(V_{\leq 0} - \tS_0)$.

Like the Sato Grassmannian, the Sato flag variety depends on the choice of base flag $V_{\leq\bullet}$, and we sometimes write $\sFl(V;V_{\leq\bullet})$ for $\sFl$.  The precise dependence is this: given two $\Z$-indexed flags $E_\bullet$ and $E'_\bullet$ of subspaces of $V$, one has $\sFl(V;E_\bullet) = \sFl(V;E'_\bullet)$ if and only if $E_\bullet \in \sFl(V;E_\bullet')$ and $E'_\bullet \in \sFl(V;E_\bullet)$.  (This is just the condition that $E_\bullet$ and $E'_\bullet$ are cofinal in both their ascending and descending sequences.)  The same condition describes when $\sGr^k(V;E_\bullet)=\sGr^k(V;E'_\bullet)$.

A bit more generally, for any increasing sequence of integers $\bp$, indexed so that $p_i \leq 0$ for $i\leq 0$ and $p_i>0$ if $i>0$, there is a {\em partial Sato flag variety}
\[
 \sFl(\bp) = \left\{ E_\bullet =( \cdots \subset E_{p_{-1}} \subset E_{p_0} \subset E_{p_1} \subset \cdots ) \,|\, E_{p_k} \in \sGr^{p_k} \right\},
\]
a subspace of $\prod_k \sGr^{p_k}$.  Its cohomology ring is naturally identified with a subring of $H_T^*\sFl = \Lambda[x;y]$, by taking polynomials that are symmetric in groups of $x$-variables $\{x_{p_k+1},\ldots,x_{p_{k+1}}\}$.  (The elementary symmetric polynomials in these variables correspond to Chern classes of $(\tS_{p_{k+1}}/\tS_{p_k})^*$.)

\begin{remark}
Our definition of $\sGr$ is the same as that of \cite[\S6]{LLS}, but our $\sFl$ is larger than theirs, which may be considered a restricted Sato flag variety, $\sFl' \subset \sFl$.  This $\sFl'$ is a union of finite-dimensional flag varieties, so its cohomology ring is an inverse limit: it is $H_T^*\sFl' = \Lambda[y][\![x]\!]_{\mathrm{gr}}$, the ring of formal series in $x$, of bounded degree, with coefficients in $\Lambda[y]$.  Pullback by the embedding $\sFl' \hookrightarrow \sFl$ corresponds to the inclusion $\Lambda[x;y] \hookrightarrow \Lambda[y][\![x]\!]_{\mathrm{gr}}$.  We prefer to work with polynomials, and hence with $\sFl$.
\end{remark}

\section{Schubert varieties and Schubert polynomials}\label{s.schub}

Schubert varieties in $\sFl$ are defined with respect to the opposite flag $V_{>\bullet}$.  For each $w\in \Sgp_\Z$, and $p,q\in\Z$, recall that
\[
  k_w(p,q) = \#\{ a\leq p \,|\, w(a)>q \}.
\]
An example is shown in Figure~\ref{f.sz}.  The \define{Schubert variety} is
\[
  \Omega_w = \{ E_\bullet\,|\, \dim( E_p \cap V_{>q}) \geq k_w(p,q) \text{ for all } p,q \}.
\]
The conventions are set up so that $\Omega_w$ is a compatible limit of similarly defined loci in the finite-dimensional varieties $\Fl(m-n,\ldots,m+n;V_{(-m,m]})$. 

The \define{Rothe diagram} and \define{essential set} of a permutation $w\in \Sgp_\Z$ are determined just as in \cite{flags}: the diagram is what remains when one strikes out boxes below and right of each dot, and the essential set is the set of $(k,p,q)$ where $(p,q)$ is a southeast corner of the diagram and $k=k_w(p,q)$.  An example is shown in Figure~\ref{f.sz}.  The conditions $\dim(E_p\cap V_{>q})\geq k$, for $(k,p,q)$ in the essential set of $w$, suffice to define $\Omega_w$; this follows from the analogous statement for finite-dimensional Schubert varieties.

\begin{figure}
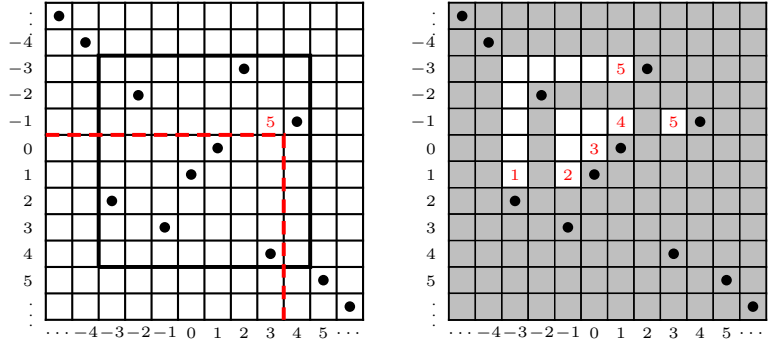

\begin{center}
\pspicture(-70,-70)(70,70)

\psgrid[unit=10pt,subgriddiv=0,gridlabels=0](-6,-6)(6,6)

{\tiny
\rput[r](-65,-55){$\vdots$}
\rput[r](-65,-45){$5$}
\rput[r](-65,-35){$4$}
\rput[r](-65,-25){$3$}
\rput[r](-65,-15){$2$}
\rput[r](-65,-5){$1$}
\rput[r](-65,5){$0$}
\rput[r](-65,15){$-1$}
\rput[r](-65,25){$-2$}
\rput[r](-65,35){$-3$}
\rput[r](-65,45){$-4$}
\rput[r](-65,55){$\vdots$}

\rput[c](-55,-65){$\cdots$}
\rput[c](-45,-65){$-4$}
\rput[c](-35,-65){$-3$}
\rput[c](-25,-65){$-2$}
\rput[c](-15,-65){$-1$}
\rput[c](-5,-65){$0$}
\rput[c](5,-65){$1$}
\rput[c](15,-65){$2$}
\rput[c](25,-65){$3$}
\rput[c](35,-65){$4$}
\rput[c](45,-65){$5$}
\rput[c](55,-65){$\cdots$}

}

\psline[linewidth=1.5pt]{-}(-40,-40)(40,-40)(40,40)(-40,40)(-40,-40)

\pscircle*(-55,55){2}
\pscircle*(-45,45){2}
\pscircle*(-35,-15){2}
\pscircle*(-25,25){2}
\pscircle*(-15,-25){2}
\pscircle*(-5,-5){2}
\pscircle*(5,5){2}
\pscircle*(15,35){2}
\pscircle*(25,-35){2}
\pscircle*(35,15){2}
\pscircle*(45,-45){2}
\pscircle*(55,-55){2}

\rput(25,15){\textcolor{red}{\tiny $5$}}
\psline[linewidth=1.5pt,linestyle=dashed,linecolor=red]{-}(-60,10)(30,10)(30,-60)

\endpspicture
\hspace{10pt}
\pspicture(-70,-70)(70,70)

\psgrid[unit=10pt,subgriddiv=0,gridlabels=0](-6,-6)(6,6)

{\tiny
\rput[r](-65,-55){$\vdots$}
\rput[r](-65,-45){$5$}
\rput[r](-65,-35){$4$}
\rput[r](-65,-25){$3$}
\rput[r](-65,-15){$2$}
\rput[r](-65,-5){$1$}
\rput[r](-65,5){$0$}
\rput[r](-65,15){$-1$}
\rput[r](-65,25){$-2$}
\rput[r](-65,35){$-3$}
\rput[r](-65,45){$-4$}
\rput[r](-65,55){$\vdots$}

\rput[c](-55,-65){$\cdots$}
\rput[c](-45,-65){$-4$}
\rput[c](-35,-65){$-3$}
\rput[c](-25,-65){$-2$}
\rput[c](-15,-65){$-1$}
\rput[c](-5,-65){$0$}
\rput[c](5,-65){$1$}
\rput[c](15,-65){$2$}
\rput[c](25,-65){$3$}
\rput[c](35,-65){$4$}
\rput[c](45,-65){$5$}
\rput[c](55,-65){$\cdots$}

}


\graybox(-60,50)
\graybox(-50,50)
\graybox(-40,50)
\graybox(-30,50)
\graybox(-20,50)
\graybox(-10,50)
\graybox(0,50)
\graybox(10,50)
\graybox(20,50)
\graybox(30,50)
\graybox(40,50)
\graybox(50,50)

\graybox(-60,40)
\graybox(-50,40)
\graybox(-40,40)
\graybox(-30,40)
\graybox(-20,40)
\graybox(-10,40)
\graybox(0,40)
\graybox(10,40)
\graybox(20,40)
\graybox(30,40)
\graybox(40,40)
\graybox(50,40)

\graybox(-60,30)
\graybox(-50,30)
\graybox(10,30)
\graybox(20,30)
\graybox(30,30)
\graybox(40,30)
\graybox(50,30)

\graybox(-60,20)
\graybox(-50,20)
\graybox(-30,20)
\graybox(-20,20)
\graybox(-10,20)
\graybox(0,20)
\graybox(10,20)
\graybox(20,20)
\graybox(30,20)
\graybox(40,20)
\graybox(50,20)

\graybox(-60,10)
\graybox(-50,10)
\graybox(-30,10)
\graybox(10,10)
\graybox(30,10)
\graybox(40,10)
\graybox(50,10)

\graybox(-60,0)
\graybox(-50,0)
\graybox(-30,0)
\graybox(0,0)
\graybox(10,0)
\graybox(20,0)
\graybox(30,0)
\graybox(40,0)
\graybox(50,0)

\graybox(-60,-10)
\graybox(-50,-10)
\graybox(-30,-10)
\graybox(-10,-10)
\graybox(0,-10)
\graybox(10,-10)
\graybox(20,-10)
\graybox(30,-10)
\graybox(40,-10)
\graybox(50,-10)

\graybox(-60,-20)
\graybox(-50,-20)
\graybox(-40,-20)
\graybox(-30,-20)
\graybox(-20,-20)
\graybox(-10,-20)
\graybox(0,-20)
\graybox(10,-20)
\graybox(20,-20)
\graybox(30,-20)
\graybox(40,-20)
\graybox(50,-20)

\graybox(-60,-30)
\graybox(-50,-30)
\graybox(-40,-30)
\graybox(-30,-30)
\graybox(-20,-30)
\graybox(-10,-30)
\graybox(0,-30)
\graybox(10,-30)
\graybox(20,-30)
\graybox(30,-30)
\graybox(40,-30)
\graybox(50,-30)

\graybox(-60,-40)
\graybox(-50,-40)
\graybox(-40,-40)
\graybox(-30,-40)
\graybox(-20,-40)
\graybox(-10,-40)
\graybox(0,-40)
\graybox(10,-40)
\graybox(20,-40)
\graybox(30,-40)
\graybox(40,-40)
\graybox(50,-40)

\graybox(-60,-50)
\graybox(-50,-50)
\graybox(-40,-50)
\graybox(-30,-50)
\graybox(-20,-50)
\graybox(-10,-50)
\graybox(0,-50)
\graybox(10,-50)
\graybox(20,-50)
\graybox(30,-50)
\graybox(40,-50)
\graybox(50,-50)

\graybox(-60,-60)
\graybox(-50,-60)
\graybox(-40,-60)
\graybox(-30,-60)
\graybox(-20,-60)
\graybox(-10,-60)
\graybox(0,-60)
\graybox(10,-60)
\graybox(20,-60)
\graybox(30,-60)
\graybox(40,-60)
\graybox(50,-60)

\pscircle*(-55,55){2}
\pscircle*(-45,45){2}
\pscircle*(-35,-15){2}
\pscircle*(-25,25){2}
\pscircle*(-15,-25){2}
\pscircle*(-5,-5){2}
\pscircle*(5,5){2}
\pscircle*(15,35){2}
\pscircle*(25,-35){2}
\pscircle*(35,15){2}
\pscircle*(45,-45){2}
\pscircle*(55,-55){2}

\rput(5,35){\textcolor{red}{\tiny $5$}}
\rput(5,15){\textcolor{red}{\tiny $4$}}
\rput(25,15){\textcolor{red}{\tiny $5$}}
\rput(-5,5){\textcolor{red}{\tiny $3$}}
\rput(-35,-5){\textcolor{red}{\tiny $1$}}
\rput(-15,-5){\textcolor{red}{\tiny $2$}}

\endpspicture

\end{center}

\caption{The permutation $w$ in $\Sgp_\Z$ given in one-line notation as $[2,-2,3,1,0,-3,4,-1]$.  The value of the rank function $k_w(3,-1)= 5$ is illustrated as the number of dots enclosed by the dashed line, at left.  The diagram and essential set are shown at right. \label{f.sz}}
\end{figure}

Schubert varieties in $\sGr$ are defined similarly, by
\[
  \Omega_\lambda = \{ E \, |\, \dim(E \cap V_{>\lambda_k-k}) \geq k \text{ for all }k\},
\]
for a partition $\lambda = (\lambda_1\geq \cdots \geq \lambda_s\geq 0)$.  As usual, it suffices to impose such conditions for $1 \leq k \leq s$, or even for those $k$ such that $\lambda_k>\lambda_{k+1}$ (since corners of the Young diagram determine the essential conditions).  These conditions also define the Schubert variety $\Omega_{w_\lambda} \subseteq \sFl$, where $w_\lambda$ is the Grassmannian permutation associated to $\lambda$.

By taking limits of finite-dimensional varieties, there is a well-defined class $[\Omega_w]$ in $H_{T}^*\sFl$.

\begin{definition}
The \define{enriched Schubert polynomial} $\sS_w(c;x;y)$ is the (unique) polynomial representing the class of the Schubert variety $\Omega_w\subseteq \sFl$.  
That is, 
\[
 \sS_w(c;x;y) = [\Omega_w]
\]
in $\Lambda[x;y] = H_{T}^*\sFl$, by definition.
\end{definition}

The enriched Schubert polynomials, by definition, are polynomials in $c$, $x$, and $y$.  Also by definition, if $m$ and $m'$ are large enough so that $w$ fixes all integers outside of $(-m,m']$, the polynomial $\sS_w(c;x;y)$ restricts to a Schubert class in the finite-dimen{\-}sional flag variety $\Fl(V_{(-m,m']})$.  So for $w\in \Sgp_{(-m,m']}$, the polynomial $\enS_w(c;x;y)$ depends only on $x_i$ and $y_i$ for $-m< i\leq m'$.  Furthermore, the (Lascoux-Sch\"utzenberger) Schubert polynomials $\SSS_v(x;-y)$ give formulas for these Schubert classes, and this proves the following:

\begin{proposition}
Suppose $w \in \Sgp_{m'}$.  Then
\[
  \enS_w( c^{(m)}; x; y) = \SSS_{1^m \times w}(x_{-m+1},\ldots,x_{m'};-y_{-m+1},\ldots,-y_{m'}),
\]
where $c^{(m)} = \prod_{i=-m+1}^0 \frac{1+y_i}{1-x_i}$.
\end{proposition}

\noindent
(The fact that the right-hand side is supersymmetric in the non-positive $x$ and $y$ variables, and therefore may be written in terms of $c^{(m)}$ variables, can be found in \cite[Corollary~2.5]{buch-rimanyi}.)

For example, if $k>0$ we have
\begin{align*}
 \enS_{s_k}(c^{(m)};x;y) &= x_{-m+1}+\cdots+x_k + y_{-m+1} + \cdots + y_k \\
 &= \SSS_{s_{m+k}}(x_{-m+1},\ldots,x_k;-y_{-m+1},\ldots,-y_k).
\end{align*}

For general $w\in \Sgp_\Z$, one can use translation operators to relate $\sS_w$ to $\SSS_v$, for some $v \in \Sgp_+$, as in \S\ref{ss.shift}.  (See also \cite{LLS,part1}.)

The proposition shows that the enriched Schubert polynomials $\sS_w(c;x;y)$ agree with the {\em back stable Schubert polynomials} $\bsS_w(x;-y)$ of \cite{LLS}.\footnote{To do this, one interprets $c = \prod_{i\leq 0}\frac{1+y_i}{1-x_i}$.  This series interpretation is not logically necessary for us, and we generally avoid it, since it assigns a double role to non-positive $x$ and $y$ variables.}

The {\em inverse} formula
\begin{equation}\label{e.inverse}
  \enS_{w}(c;x;y) = \enS_{w^{-1}}(\omega(c);y;x),
\end{equation}
where $\omega(c) = 1/(1-c_1+c_2-\cdots)$, follows by transposing the flags in the definition of $\Omega_w$; see  \cite[Proposition~1.2]{part1}.

Finite-dimensional Schubert classes form $\Z[y]$-module bases for each cohomology ring $H_{T}^*\Fl(m-n,\ldots,m+n;V_{(-m,m]})$.  So in the limit, the classes of $\Omega_w\subseteq \sFl$ form a $\Z[y]$-basis for $H_{T}^*\sFl$.  (As usual, one may think about compatible sequences of finite-dimensional Schubert varieties instead.)  It follows that the polynomials $\sS_w$ form a basis for $\Lambda[x;y]$ over $\Z[y]$, as $w$ ranges over $\Sgp_\Z$.  In fact, these considerations prove a more refined statement:
\begin{proposition}\label{p.variables}
Fix positive integers $n,n'$.
\begin{enumerate}[(i)]
 \item If $w(i)<w(i+1)$ for all $i<n$ and all $i>n'$, then $\enS_w(c;x;y)$ lies in the subalgebra $\Z[c,y][x_{-n+1},\ldots,x_{n'}]$.  As $w$ varies over such permutations, the enriched Schubert polynomials $\enS_w(c;x;y)$ form a basis for this subalgebra, considered as a module over $\Z[y]$.
 
 \item If $w^{-1}(i)<w^{-1}(i+1)$ for all $i<n$ and all $i>n'$, then $\enS_w(c;x;y)$ lies in the subalgebra $\Z[c,x][y_{-n+1},\ldots,y_{n'}]$.  As $w$ varies over such permutations, the enriched Schubert polynomials $\enS_w(c;x;y)$ form a basis for this subalgebra, considered as a module over $\Z[x]$.
\end{enumerate}
\end{proposition}

\noindent
(The first statement is proved by considering the Schubert basis for the partial flag variety $\Fl(\bp)$, where $\bp = (-n+1,\ldots,n')$.  The second statement is equivalent to the first by applying the inverse formula \eqref{e.inverse}.)

For Chern series $c$, $c'$, and $\bbc$ with $\bbc=c\cdot c'$, there is a Cauchy formula
\begin{equation}\label{e.cauchy}
  \enS_w(\bbc;x;y) = \sum_{vu \dot{=} w} \enS_u(c;x;t)\, \enS_v(c';-t;y),
\end{equation}
where $vu\dot{=}w$ means $vu=w$ and $\ell(u)+\ell(v)=\ell(w)$ \cite{LLS,LLS2,part1}.

Following \cite[\S4.6]{LLS}, by specializing $x_i=-y_i$ for all $i$ one obtains the \define{double Stanley polynomials}
\begin{equation}\label{e.stanley}
  F_w(c;y) = \enS_w(c;-y;y).
\end{equation}
More generally, there are polynomials $F_w^v(c;y) = \enS_w(c;-y^v;y)$ obtained by specialization $x_i=-y_{v(i)}$.  
Further specializing the $y$ variables to zero recovers the ``stable Schubert'' formulation of the Stanley symmetric functions, $F_w(c) = \enS_w(c;0;0)$.

For Grassmannian permutations, the Schubert polynomials have a determinantal (Kempf-Laksov) formula:
\begin{align} \label{e.kempf-laksov}
 \enS_{w_\lambda}(c;x;y) &= \det( c(i)_{\lambda_i-i+j} )_{1\leq i,j\leq s} 
\end{align}
where
\begin{align*}
  c(i) &= c\cdot \frac{\prod_{j\leq \lambda_i-i}( 1+y_i ) }{\prod_{j\leq 0} (1+y_i) } \\
       &= c \cdot c^T(V_{\leq \lambda_i-i} - V_{\leq 0}).
\end{align*}
These evaluate to double Schur functions\footnote{Under the evaluation $c= \prod_{i\leq 0}\frac{1+y_i}{1-x_i}$, some authors write these as $s_{\lambda}(x/y|\!|{-y})$, notation we avoid in the present context.} $s_\lambda(c|{y})$, with \eqref{e.kempf-laksov} becoming a variation of the Jacobi-Trudi formula.

More generally, any {\it vexillary} permutation $w=w(\triple)$ in $\Sgp_\Z$ has a similarly explicit determinantal formula (see \cite{part1}); for example, for any $m<n$, the permutation $w_{\circ}^{(m,n)} = [n,n-1,\ldots,m]$ is vexillary.  Any $w\in \Sgp_\Z$ lies in $\Sgp_{[m,n]}$ for some $m<n$, so $w\leq w_{\circ}^{(m,n)}$.  Any enriched Schubert polynomial may therefore be computed from the explicit formula for $\enS_{w_\circ^{(m,n)}}$ using the divided difference recursion
\[
  \partial^x_i\enS_w = \begin{cases} \enS_{ws_i} & \text{if } ws_i<w; \\ 0& \text{if }ws_i>w. \end{cases}
\]
Here $\partial^x_i$ is the usual divided difference operator acting on $x$ variables, so for any $f\in \Z[c,x,y]$ and any $i\in\Z$,
\[
  \partial^x_i f = \frac{f(\ldots,x_i,x_{i+1},\ldots)-f(\ldots,x_{i+1},x_i,\ldots)}{x_i-x_{i+1}}.
\]

In what follows, we study further algebraic properties of the polynomials $\enS_w$ using the geometry of $\sFl$.

\section{Degeneracy loci}

The enriched Schubert polynomials represent classes of degeneracy loci.  By taking a sufficiently general base variety $X$, they may be characterized uniquely by this property.  Precedents for the setup we consider can be traced to \cite{flags}, and especially \cite{buch-fulton}.

On a nonsingular variety $X$, we have a vector bundle $V$ of rank $m+n$, with flags
\[
  E_\bullet: \cdots \subset E_{-1}\subset E_{0}\subset E_1 \subset E_2 \subset \cdots \subset V
\]
and
\[
  F_\bullet: \cdots \subset F_{1}\subset F_{0}\subset F_{-1} \subset F_{-2} \subset \cdots \subset V,
\]
indexed so that $\rk E_0 = \rk F_0 = m$.  (So $\rk E_p = m+p$ and $\rk F_q = m-q$.)

For $w\in \Sgp_{(-m,n]}$, there is a degeneracy locus
\[
  D_w(E_\bullet \cap F_\bullet) = \{x\in X\,|\, \dim(E_p\cap F_q) \geq k_w(p,q) \text{ for all }p,q \}
\]
in $X$.  As usual, it suffices to impose conditions $\dim(E_p\cap F_q) \geq k$ for $(k,p,q)$ in the essential set.

\begin{theorem}\label{t.intersection}
Assume $D_w(E_\bullet\cap F_\bullet) \subseteq X$ has codimension $\ell(w)$.  Under the evaluations
\[
  c\mapsto c(V-E-F), \quad x_i \mapsto -c_1(E_i/E_{i-1}), \quad y_i \mapsto c_1(F_{i-1}/F_i),
\]
the enriched Schubert polynomial $\sS_w(c;x;y)$ maps to the class $[D_w(E_\bullet\cap F_\bullet)]$ in $H^*X$.
\end{theorem}

This is proved in \cite{part1}.  It can also be deduced directly from the formula for $[\Omega_w]$, as follows.  Choose an approximation of the classifying space $\mathds{B}$ for $T$ so that the vector bundle $V$ and flag $F_\bullet$ are pulled back from tautological bundles on $\mathds{B}$, and $F_q$ is the pullback of $V_{>q}$.  Take the flag bundle $\sFl \to \mathds{B}$ over that classifying space, constructing $f\colon X \to \sFl$ so that $E_\bullet$ is pulled back from the tautological $\tS_\bullet$.  Then $D_w(E_\bullet \cap F_\bullet) = f^{-1}\Omega_w$.  More details appear in \cite[Chapters 11--12]{ecag}.

\section{Fixed points}\label{s.fp}

Recall that $T=\prod_{i\in \Z} \C^*$ acts on $V$ by scaling coordinates.  To describe the $T$-fixed points of the various infinite flag varieties, we need to say more about permutations of $\Z$.

First, for any sets $X$ and $Y$, let $\Inj(X,Y)$ be the set of all injections from $X$ into $Y$, and let $\Inj(X)$ be the monoid of injections from $X$ into itself.  So $\Bij(X)\subset \Inj(X)$ is a subgroup.

As usual, we are concerned with subsets of $\Z$.  The submonoid $\Inj^0(\Z)\subset \Inj(\Z)$ consists of all $w$ such that
\[
 \#\{i\leq 0 \,|\, w(i)>0\} = \#\{i> 0 \,|\, w(i)\leq0\},
\]
and both these sets are finite.  (That is, $w$ has finitely many sign changes, and they are balanced.)  
Any $w\in \Inj^0(\Z)$  also has $\#\{i\leq k \,|\, w(i)>0\} - \#\{i> k \,|\, w(i)\leq0\}=k$ for any integer $k$.

The set $\Inj(\Z_{>0})$ may be constructed as the inverse limit of $\Inj( [1,n], \Z_{>0} )$ over $n>0$.  This mirrors the construction of $\Fl_+(V_{>0})$, and shows that the $T$-fixed points of $\Fl_+(V_{>0})$ are indexed by $w\in \Inj(\Z_{\leq 0})$: they are precisely the flags determined by the ordered bases $e_{w(1)}, e_{w(2)}, \ldots$, so the $k$-dimensional component is the span of $e_{w(i)}$ for $1\leq i\leq k$.

Similarly, the $T$-fixed points of $\Fl_-(V_{\leq0})$ are indexed by $w\in \Inj(\Z_{\leq 0})$, so the codimension $k$ component is defined by $e_{w(i)}^*=0$ for $-k<i\leq 0$.  Equivalently, it is the span of $e_{w(i)}$ for $i\leq k$, together with all $e_i$ for $i\leq 0$ not in the image of $w$.  So the flag varieties $\Fl_+$ and $\Fl_-$ have uncountably many fixed points.

The fixed points of the Sato Grassmannian $\sGr$, on the other hand, are (countably) indexed by partitions $\lambda$, or equivalently by Grassmannian elements $w_\lambda\in \Sgp_\Z$.  The fixed subspace corresponding to $\lambda$ is spanned by $e_{w_\lambda(i)}$ for $i\leq0$.    (See also \cite[\S 7]{pressley-segal}.)

The fixed points of the Sato flag variety $\sFl$ are indexed by $w\in \Inj^0(\Z)$.  A fixed flag is determined by the ordered basis $\ldots,e_{w(-1)},e_{w(0)},e_{w(1)},\ldots$, so its $k$th component is the span of $e_{w(i)}$ for $i\leq k$, together with all $e_i$ for $i\leq 0$ not in the image of $w$.

The formula defining $k_w(p,q)$ works verbatim for any $w\in \Inj^0(\Z)$, because the set it enumerates is finite for such $w$.  Using this, one can extend the definition of Bruhat order from $\Sgp_\Z$ to $\Inj^0(\Z)$.

Generally, we write $p_w$ for the point corresponding to a fixed flag, also using $p_\lambda = p_{w_\lambda}$ for points in $\sGr$.

From the definitions of Schubert varieties and Bruhat order, one sees that
\[
  p_v \in \Omega_w \quad \text{ iff } \quad v \geq w.
\]
Here, as usual, we assume $w\in\Sgp_\Z$, but $v$ varies over $\Inj^0(\Z)$.

Formulas for restricting a Schubert class to a fixed point follow from the finite-dimensional case.  We have
\begin{align} \label{e.interpolate1}
  [\Omega_w]|_{p_w} &= \mathop{\prod_{i<j}}_{w(i)>w(j)} (y_{w(i)}-y_{w(j)})
\end{align}
and, for any $v\in\Inj^0(\Z)$,
\begin{align}\label{e.interpolate2}
 [\Omega_w]|_{p_v} &= 0 \quad \text{ if } v\not\geq w.
\end{align}

For $v\in \Inj^0(\Z)$, let
\[
 c^v = \mathop{\prod_{i\leq0, v(i)>0}}_{j>0, v(j)\leq0}\frac{ 1+y_{v(j)} }{ 1+y_{v(i)} } \quad \text{and}\quad y^v_i = y_{v(i)}.
\]
(Note that $c^v$ is a finite product.)

\begin{proposition}
The enriched Schubert polynomial $\sS_w(c;x;y)$ satisfies the specialization formulas
\[
  \sS_w(c^w;-y^w;y) = \mathop{\prod_{i<j}}_{w(i)>w(j)} (y_{w(i)}-y_{w(j)})
\]
and, for $v\in\Inj^0(\Z)$,
\[
  \sS_w(c^v;-y^v;y) = 0 \quad \text{ if } v\not\geq w.
\]
These properties, as $v$ ranges over $\Sgp_\Z$, determine $\sS_w(c;x;y)$ uniquely.
\end{proposition}

The fact that these properties are satisfied follows from the corresponding properties of Schubert classes.  The proof that they uniquely determine a Schubert class also follows from the finite-dimensional case, by taking a sufficiently large approximation.  One only needs to let $v$ vary over $\Sgp_\Z$ (rather than all fixed points), because specializations of $\sS_w(c;x;y)$, involving only finitely many variables, are insensitive to the difference between $\Sgp_\Z$ and $\Inj^0(\Z)$.

\begin{remark}
Using the identification with $T$-fixed points of $\sFl$, the topology induced on $\Inj^0(\Z)$ is not discrete, but rather a limit of discrete sets.  The subgroup $\Sgp_\Z \subset \Inj^0(\Z)$ is dense, and this is another reason that fixed points indexed by $\Sgp_\Z$ suffice to determine Schubert polynomials.
\end{remark}

\begin{remark}
Later we will need to consider smaller torus actions.  Just as for finite-dimen{\-}sional flag varieties, such actions may have larger fixed loci.  In particular, we will use $T$ acting diagonally on $\bV = V\oplus V$, so each weight space is $2$-dimensional.  The fixed loci for the corresponding actions on $\sGr(\bV)$ and $\sFl(\bV)$ have infinite-dimensional components.
\end{remark}

\section{Duality, projection, and shift morphisms}

A major advantage of working with $\sGr$ and $\sFl$ is that new morphisms become evident.  As usual, these can also be described using only finite-dimensional varieties, but it is often clearer to think about the infinite flag varieties.

\subsection{Duality}\label{ss.duality}

Fix a linear isomorphism $f\colon V \xrightarrow{\sim} V^{*'}$, where as before $V^{*'} \subset V^*$ is the restricted dual.  For any subspace $E\subseteq V$, one has the associated orthogonal complement
\[
  E^\perp = \{ v\in V \,|\, f(u)(v) = 0 \text{ for all } u \in E \}.
\]
This operation reverses inclusion, so the image of the standard flag is given by the spaces $V^\perp_{\leq -k}$.

There is a \define{duality morphism}
\[
  \sGr^k(V; V_{\leq\bullet} ) \to \sGr^{-k}(V; V^\perp_{\leq-\bullet}),
\]
by $E \mapsto E^\perp$.

The same formula defines an automorphism of $\sFl(V)$, sending a flag with components $E_k$ to one with components $E_{-k}^\perp$.

From now on, we assume the isomorphism $f\colon V \to V^{*'}$ is given by the skew-symmetric form sending $e_i \mapsto  e_{1-i}^*$ for $i>0$, and $e_i \mapsto -e_{1-i}^*$ for $i\leq 0$.  In this case, the duality morphism is an involution, equivariant with respect to the automorphism of $T$ defined on characters by $y_i \mapsto -y_{1-i}$, and the standard flag is preserved, with $(V_{\leq k})^\perp = V_{\leq -k}$.  (All of this holds as well for a symmetric form.)

The induced automorphism $\omega$ of $H_T^*\sFl =\Lambda[x;y]$ is given by
\[
  \omega(c) = 1/(1-c_1+c_2-\cdots), \quad \omega(x_i) = -x_{1-i}, \quad \omega(y_i)=-y_{1-i}.
\]
The same notation is used for the automorphism of $\Sgp_\Z$, defined by $\omega(w)(i) = 1-w(1-i)$.  One checks that $k_{\omega(w)}(p,q) = k_w(-p,-q)$, so the duality morphism sends $\Omega_w$ to $\Omega_{\omega(w)}$.  It follows that
\[
  \omega(\enS_w(c;x;y)) = \enS_{\omega(w)}(c;x;y).
\]

Following \cite{LLS}, one defines $\SSS_w(x;y)$ for any $w\in \Sgp_{\neq 0}$ using the duality involution: for $w=w_-\cdot w_+$, with $w_-\in \Sgp_-$ and $w_+\in \Sgp_+$, one defines $\SSS_w = \omega(\SSS_{\omega(w_-)})\cdot \SSS_{w_+}$.

\subsection{Projections}
For each $k$, there is a \emph{projection} $\pi_k\colon \sFl \to \sGr^k$, sending $E_\bullet$ to $E_k$.  This is a fiber bundle, and the fiber over $V_{\leq k}\in \sGr^k$ is $\Fl_-(V_{\leq k})\times \Fl_+(V_{>k})$.  In particular, the inclusion $\Lambda[y] \hookrightarrow \Lambda[x;y]$ corresponds to $\pi_0^*$, and the homomorphism
\[
  \Lambda[x;y] \to \Z[x;y], \qquad c\mapsto 1
\]
corresponds to restriction to the fiber over $V_{\leq 0}\in \sGr$.

\begin{proposition}
If $w\in \Sgp_\Z$ is not in $\Sgp_{\neq0}$, then $\enS_w(1;x;y)=0$.  If $w=w_+\cdot w_-\in \Sgp_{\neq0}$, then $\enS_w(1;x;y)=\SSS_{w}(x;-y)$.
\end{proposition}

\begin{proof}
For the first statement, we show that $\Omega_w \cap \pi_0^{-1}(V_{\leq0})$ is empty.  It suffices to show the fixed-point sets of $\Omega_w$ and $\pi_0^{-1}(V_{\leq0})$ are disjoint.  Since $w\not\in \Sgp_{\neq0}$, at least one $i\leq0$ has $w(i)>0$.  That is, $k_w(0,0)>0$.  The fixed points in $\pi_0^{-1}(V_{\leq0}) = \Fl_-(V_{\leq 0})\times \Fl_+(V_{>0})$ are $p_v$, for $\Inj(\Z_{\leq 0}) \times \Inj(\Z_{>0})$.  Each such $v$ has $k_v(0,0)=0$.  So $v\not\geq w$, and therefore $p_v \not\in \Omega_w$.

The second statement follows from the fact that  $\Omega_w \cap \pi_0^{-1}(V_{\leq0}) = \Omega_{w_-} \times \Omega_{w_+}$ inside $\pi_0^{-1}(V_{\leq0}) = \Fl_-(V_{\leq 0}) \times \Fl_+(V_{>0})$, together with the definition of $\SSS_w$.
\end{proof}

\subsection{Shift}\label{ss.shift}
Let $\sh \colon V \to V$ be the linear automorphism given by $e_i \mapsto e_{i-1}$.  This induces \emph{shift morphisms}, also written $\sh\colon \sGr^{k} \to \sGr^{k-1}$, sending $E\subset V$ to $\sh(E)\subset V$, and an automorphism $\sh\colon \sFl \to \sFl$, defined by $\sh(E_\bullet)_k = \sh(E_{k+1})$.  The shift morphisms are equivariant with respect to a similar automorphism of $T = \prod_{i\in\Z} \C^*$, sending $z_i \mapsto z_{i-1}$.

To construct the shift morphism from finite-dimensional varieties, one uses the system of maps
\begin{align*}
  \Gr(m+k, V_{(-m,m]}) &\hookrightarrow \Gr(m+k,V_{(-m-1,m+1]}) \\
   (E\subset V_{(-m,m]}) &\mapsto ( \sh(E) \subset V_{(-m-1,m+1]}).
\end{align*}
Taking the union over $m$ on each side determines a morphism $\sGr^k \rightarrow \sGr^{k-1}$.

Pullback by the shift morphism gives the \emph{translation operator} $\gamma \colon \Lambda[x;y] \to \Lambda[x;y]$ on cohomology.  Explicitly, $\gamma=\sh^*$ is given by
\begin{align*}
  \gamma(x_i) &= x_{i+1},\\
  \gamma(y_i) &= y_{i+1}, \text{ and}\\
  \gamma(c_k)  &= \sum_{p=0}^k c_p\,x_1^{k-p} + y_1\sum_{p=0}^{k-1} c_{p}\,x_1^{k-1-p}.
\end{align*}
(The action on $c$ variables can be written concisely as $\gamma(c) = c\cdot \frac{1+y_1}{1-x_1}$.)  The action on $x$ variables comes from $\sh^*(\tS_i) = \tS_{i+1}$, and the $y$ variables are determined by the automorphism of $T$.  For the $c$ variables, one observes $\sh^*(V_{\leq 0}) = V_{\leq 1}$, so
\[
 \sh^*c^T(V_{\leq0}-\tS_0) = c^T(V_{\leq 1} - \tS_1) = c^T(V_{\leq0}-\tS_0)\cdot c^T(\C\cdot e_1 - \tS_1/\tS_0).
\]

The homomorphism $\gamma$ is invertible.  For any $m\in\Z$, one has $\gamma^m(x_i)=x_{i+m}$ and $\gamma^m(y_i)=y_{i+m}$, with the action on $c$ variables determined by
\begin{align*}
  \gamma^m(c)  &= \begin{cases} c\cdot \prod_{i=1}^m\frac{ 1+y_i }{ 1-x_i } & \text{if } m\geq 0; \\ c\cdot \prod_{i=m+1}^0\frac{ 1-x_i }{ 1+y_i } & \text{if } m< 0. \end{cases}
\end{align*}

For any $w\in\Inj(\Z)$, the injection $\gamma^m(w)$ is defined by $\gamma^m(w)(i) = m+w(i-m)$.

\begin{proposition}
We have $\gamma^m( \sS_w(c;x;y) ) = \sS_{\gamma^m(w)}(c;x;y)$, for any $m\in\Z$ and $w\in\Sgp_\Z$.
\end{proposition}

\begin{proof}
The diagram of $\gamma(w)$ is obtained from that of $w$ by shifting one unit in the southeast direction; in particular, $k_{\gamma(w)}(p+1,q+1) = k_w(p,q)$.  Since $\sh^*(\tS_p)=\tS_{p+1}$ and $\sh^*V_{>q}=V_{>q+1}$, it follows that $\sh^{-1}\Omega_w = \Omega_{\gamma(w)}$ and therefore $\sh^*[\Omega_w]=[\Omega_{\gamma(w)}]$.
\end{proof}

\section{Direct sum morphism and coproduct}\label{s.dirsum}

We will define and study a direct sum morphism
\[
\boxplus \colon \sGr^k(V) \times \sGr^l(V) \to \sGr^{k+l}(\bV),
\]
as well as a similar one for flag varieties, giving an algebraic version of an $H$-space structure on $\sGr$.  We pay special attention to the action of these morphisms on Schubert classes.

Here $V$ is our usual vector space, with basis $e_i$ for $i\in \Z$, and $\bV = V\oplus V$.  Some care is required in the specification of base flags for $\sGr$ and $\sFl$.  We fix an ordered basis for $\bV = V\oplus V$ by vectors $\bbe_i$, for $i\in \frac{1}{2}\Z$.  These are
\[
  \bbe_i = \begin{cases} (e_i,0) & \text{for }i\in\Z; \\ (0,e_{i+\frac{1}{2}}) &\text{for }i\in \Z+\frac{1}{2}. \end{cases}
\]
So $\bbe_{-\frac{1}{2}} = (0,e_0)$, $\bbe_0 = (e_0,0)$, $\bbe_{\frac{1}{2}} = (0,e_1)$, etc.  The torus $T$ acts diagonally on $\bV$, so both $\bbe_i$ and $\bbe_{i-\frac{1}{2}}$ are scaled by the character $y_i$.

Standard subspaces, indexed by subsets of $\frac{1}{2}\Z$, are defined in the evident way.  In particular, we have a standard flag $\bV_{\leq \bullet}$.  Furthermore, $\bV_{(m,m]} = V_{(m,m]}\oplus V_{(m,m]}$ and $\bV_{\leq k} = V_{\leq k} \oplus V_{\leq k}$, when $m$ and $k$ are integers.

\subsection{Grassmannians}

We will describe the setup and state some results for the Grassmannian first, and prove the more general analogues for the flag variety in the following subsection.

As before, there is an isomorphism $H_T^*\sGr(\bV) = \Lambda[y]$.  Here we use the notation $\Lambda = \Z[\bbc]=\Z[\bbc_1,\bbc_2,\ldots]$, and the map identifies $\bbc_k = c^T(\bV_{\leq 0} - \tbS_0)$, where $\tbS_0$ is the tautological bundle on $\sGr(\bV)$.  Similarly, one has $H_T^*\sFl(\bV) = \Lambda[x;y]$, with $x_i = -c_1^T( \tbS_i/\tbS_{i-1} )$.

The direct sum morphism
\[
  \boxplus\colon \sGr^k(V;V_{\leq\bullet}) \times \sGr^l(V;V_{\leq\bullet}) \to \sGr^{k+l}(\bV;\bV_{\leq\bullet})
\]
given by $\boxplus( E,F ) = E\oplus F$, is readily checked to be well-defined and $T$-equivariant.

\begin{proposition}\label{p.embisom}
The morphism
\[
  f\colon \sGr(V) \to \sGr(\bV), \qquad  E \mapsto V_{\leq 0} \oplus  E,
\]
induces the standard isomorphism $\Lambda[y] \to \Lambda[y]$ on cohomology rings, sending $\bbc_k \mapsto c_k$.
\end{proposition}

\begin{proposition}
The homomorphism
\[
  H_T^*\sGr(\bV) \xrightarrow{\boxplus^*} H_T^*(\sGr(V)\times \sGr(V)),
\]
is identified with the homomorphism of $\Z[y]$-algebras
\[
 \Lambda[y] = \Z[\bbc,y] \xrightarrow{\Delta} \Lambda[y]\otimes_{\Z[y]} \Lambda[y] = \Z[c,c',y],
\]
given by $\bbc_k \mapsto c_k + c_{k-1} c'_1 + \cdots c_1 c'_{k-1} + c'_k$.  (Here $c=c^T(V_{\leq0}-S_0)$ comes from the first factor of $\sGr(V)$, and $c'=c^T(V_{\leq0}-S_0')$ comes from the second factor, so $\bbc=c\cdot c'$.)
\end{proposition}

The first of these propositions follows from the second, after replacing $c$ by $c'$, since $f(E) = \boxplus(E,V_{\leq 0})$.  And the second proposition is simply the equation $\boxplus^*c^T(\bV_{\leq0} - \tbS_0) = c^T(V_{\leq0}+V_{\leq0}-\tS_0-\tS'_0) = c^T(V_{\leq0}-S_0)\cdot c^T(V_{\leq0}-S'_0)$.  

Using the isomorphism $H_T^*\sGr(V)=H_T^*\sGr(\bV)$, the homomorphism $\boxplus^*=\Delta$ determines a commutative coproduct structure on $H_T^*\sGr(V)$.  This coproduct has been studied by many authors.  It is induced by the coproduct on $\Lambda$, and it is well known that this can be written in the Schur basis by
\[
 \Delta(s_\lambda(\bbc)) = \sum_{\mu,\nu} c_{\mu,\nu}^\lambda s_{\mu}(c) \otimes s_{\nu}(c'),
\]
where $c_{\mu,\nu}^\lambda$ is the Littlewood-Richardson coefficient.  So it can be computed from an expression in terms of the Schur basis.  (Using \eqref{e.kempf-laksov}, the Schur function $s_\lambda(c)$ is defined as the determinant
\begin{align*}
 s_\lambda(c) = s_\lambda(c|0) = \det( c_{\lambda_i-i+k} )_{1\leq i,j \leq s}
\end{align*}
for any partition $\lambda_1\geq \cdots \geq \lambda_s \geq 0$.)

We are more interested in the Schubert basis.  Schubert varieties in $\sGr(\bV)$ are defined with respect to a flag $\bV_\bull^-$, where for $q\in\Z$,
\begin{equation}\label{e.double-flag}
  \bV^-_q = V_{>0} \oplus V_{>q} .
\end{equation}
Then $\bbOmega_\lambda = \{ E\,|\, \dim(E \cap \bV^-_{\lambda_i-i})\geq i\text{ for all }i \}$.  Under the embedding $f\colon \sGr(V) \to \sGr(\bV)$, we have $f^{-1}\bbOmega_\lambda  = \Omega_\lambda$, so $f^*[\bbOmega_\lambda] = [\Omega_\lambda]$ and 
\[
 [\bbOmega_\lambda] = \enS_{w_\lambda}(\bbc;x;y) = s_\lambda(\bbc|y).
\]

Molev gives formulas for the structure constants here \cite{molev}.  In our geometric context, we have
\[
  \boxplus^*[\bbOmega_\lambda] = \sum_{\mu,\nu} \hat{c}_{\mu,\nu}^\lambda(y) \, [\Omega_\mu] \times [\Omega_\nu],
\]
for {\em dual Littlewood-Richardson polynomials} $\hat{c}_{\mu,\nu}^\lambda(y) \in \Z[y]$.  In terms of Schubert polynomials, this is equivalent to the Cauchy formula
\begin{align*}
  \enS_{w_\lambda}(\bbc;x;y) &= \sum_{uv \dot{=} w_\lambda}  F_{u}(c;y) \cdot \enS_{v}(c';x;y) \\
  &= \sum_{\mu,\nu \subset \lambda} \hat{c}_{\mu,\nu}^\lambda(y)\,\enS_{w_\mu}(c;x;y)\cdot \enS_{w_{\nu}}(c';x;y).
\end{align*}
(See \cite[\S5]{part1} and \cite[\S4.8]{LLS}.\footnote{In the notation of \cite{LLS},  evaluating $y=-a$ and $c=\prod_{i\leq0}\frac{1-a_i}{1-x_i}$ sends $\enS_{w_\lambda}(c;x;y)$ to $s_\lambda(x|\!|a)$.  In particular our $\hat{c}^\lambda_{\mu,\nu}(y)$ is their $\hat{c}^\lambda_{\mu,\nu}(-a)$.  The translation to Molev's notation is explained in \cite[\S A.4]{LLS}.})  That is, for $u=w_\lambda w_\nu^{-1}$, the Stanley function expands as $F_u(c;y) =\sum_{\mu} \hat{c}_{\mu,\nu}^\lambda(y)\,\enS_{w_\mu}(c;x;y)$.   The polynomial $\enS_{w_\lambda}(c;x;y)=s_\lambda(c|y)$ is always independent of $x$, since it represents a class coming from $H_T^*\sGr = \Lambda[y]$.

The coefficients $\hat{c}_{\mu,\nu}^\lambda(y)$ are {\em Graham-positive}; this is a special case of \cite[Theorem~4.22]{LLS}.  We will give a proof which covers the general case below.

\begin{proposition}\label{p.graham1}
Each $\hat{c}_{\mu,\nu}^\lambda(y)$ is a nonnegative combination of terms which are products of linear factors $y_i-y_j$, for $i\succ j$, ordered so that the nonpositive indices are all greater than the positive ones.  (That is, $1\prec 2 \prec \cdots \prec -2 \prec -1 \prec 0$.)
\end{proposition}

\begin{example}
The nonzero coefficients for $\lambda=(3,1)$ are:
\begin{eqnarray*}
  &\hat{c}^{(3,1)}_{\emptyset,(3,1)} = \hat{c}^{(3,1)}_{(1),(2,1)} =  \hat{c}^{(3,1)}_{(1),(3)} = \hat{c}^{(3,1)}_{(2),(1,1)} = \hat{c}^{(3,1)}_{(2),(2)} = 1,  \\
 &\hat{c}^{(3,1)}_{(1),(2)} = y_0-y_1, \\
& \hat{c}^{(3,1)}_{(1),(1,1)} = y_2-y_1, \\
& \hat{c}^{(3,1)}_{(1),(1)} = (y_2-y_1)(y_0-y_1).
\end{eqnarray*}
One can have repeated factors, e.g., $\hat{c}^{(2,2,1)}_{(1),(1,1)} = (y_0-y_1)^2$.  In fact, we will see that only linear forms and squares of linear forms occur as factors (Theorem~\ref{t.positive}).
\end{example}

As usual, the morphism $\boxplus$ comes from compatible morphisms of finite-dimensional varieties, $\boxplus\colon \Gr(m,V_{(-m,m]}) \times \Gr(m,V_{(-m,m]}) \to \Gr(2m,\bV_{(-m,m]})$.  The subvariety $\boxplus( X_\mu \times X_\nu ) \subseteq \Gr(2m,\bV_{(-m,m]})$ is a Richardson variety, $X_{\mu \oslash_m \nu} \cap \Omega_{\rho_m}$, where $\mu \oslash_m \nu$ is the partition $(\nu_1+m,\ldots,\nu_m+m,\mu_1,\ldots,\mu_m)$, and $\rho_m$ is the $m\times m$ rectangle.  (In Young diagrams, one forms $\mu \oslash_m \nu$ by placing $\nu$ to the right of the $m\times m$ rectangle, and placing $\mu$ below the rectangle; we are assuming $m$ is at least equal to the number of parts of $\nu$ and to the largest part of $\mu$.  This is \cite[Proposition~2.1]{thomas-yong}.)  The coefficients $\hat{c}_{\mu,\nu}^\lambda$ arise in the expansion of the class of this Richardson variety in a Schubert basis with respect to a third $T$-invariant flag: the one corresponding to the ordered basis
\begin{align*}
&(e_{-m+1},0), \ldots, (e_0,0), (0,e_{-m+1}), \ldots, (0,e_0), \\
& \qquad \qquad \qquad   (0,e_1),\ldots,(0,e_m), (e_1,0), \ldots, (e_m,0).
\end{align*}

This interpretation leads to another way of computing.  Fix a sufficiently large $m$, consider variable sets $x=(x_{-2m+1},\ldots,x_{2m})$ and $t=(t_{-2m+1},\ldots,t_{2m})$, and let $s_\lambda(c|t)$ be the specialization of $\enS_{w_\lambda}(c;x;t)$ by $c = \prod_{i=-2m+1}^{0} \frac{1+t_i}{1-x_i}$.  Then $\hat{c}_{\mu,\nu}^\lambda(y)$ is the coefficient of $s_{\mu \oslash_m \nu}(c|{\mathbb{y}})$ in the expansion of $s_\lambda(c|\tilde{\mathbb{y}} ) \cdot s_{\rho_m}(c|{\mathbb{y}})$, where
\begin{align}
\mathbb{y} &= (y_{-m+1},\ldots,y_m,y_{-m+1},\ldots,y_m) \label{e.y1} \intertext{and} 
\tilde{\mathbb{y}} &= (y_{-m+1},\ldots,y_0,y_{-m+1},\ldots,y_0,y_1,\ldots,y_m,y_1,\ldots,y_m). \label{e.y2}
\end{align}
For example, $\hat{c}^{(2,2,1)}_{(1),(1,1)}(y) = (y_0-y_1)^2$ is the coefficient of $s_{(3,3,1)}(c|{\mathbb{y}})$ in the product
\[
  s_{(2,2,1)}(c| y_{-1},y_0,y_{-1},y_0,y_1,y_2,y_1,y_2) \cdot s_{(2,2)}(c|y_{-1},y_0,y_1,y_2,y_{-1},y_0,y_1,y_2).
\]
(In comparison with \cite{LLS}, our $s_\lambda(c|t)$ is their $s_\lambda(x|\!|{-a})$.)

\subsection{Flag varieties}

The direct sum morphism extends to an action on the flag variety: one defines
\[
  \boxplus\colon  \sGr(V)\times \sFl(V) \to \sFl(\bV)
\]
in the same way, so that $(F,E_\bullet)$ is sent to the flag $\mathds{E}_\bullet$ with $\mathds{E}_k = F\oplus E_k$.  The pullback $\boxplus^*\colon H_T^*\sFl \to H_T^*(\sGr\times\sFl)$ is identified with a co-module operation $\Delta\colon \Lambda[x;y] \to \Lambda[y] \otimes_{\Z[y]} \Lambda[x;y]$.  As before, this homomorphism is determined by its values on Schur polynomials, and one can compute using classical Littlewood-Richardson numbers; but also as before, we are more interested in the behavior of Schubert polynomials.

The morphism $\boxplus$ induces an embedding $f\colon \sFl(V) \hookrightarrow \sFl(\bV)$, by $E_\bullet\mapsto V_{\leq0}\oplus E_\bullet$, and as in Proposition~\ref{p.embisom}, the pullback is an isomorphism on cohomology rings, $\Lambda[x,y] \mapsto \Lambda[x,y]$, sending $\bbc\to c$.

Schubert varieties in $\sFl(\bV)$ are again defined with respect to the flag $\bV^-_\bullet$ described in \eqref{e.double-flag}, so
\[
  \bbOmega_w = \big\{ E_\bullet \,|\, \dim(E_p \cap \bV^-_q) \geq k_w(p,q) \text{ for all }p,q \big\}.
\]
As before, $f^{-1}\bbOmega_w = \Omega_w$, and we have $[\bbOmega_w] = \enS_w(\bbc;x;y)$ in $H_T^*\sFl(\bV)$.

The action on Schubert classes is by
\[
  \boxplus^*[\bbOmega_w] = \sum_{\mu,v} \hat{c}^w_{\mu,v}(y) [\Omega_\mu] \times [\Omega_v].
\]
Using $\bbc=c\cdot c'$, this is expressed via the Cauchy formula as
\begin{align*}
 \enS_w(\bbc;x;y) &= \sum_{uv \dot{=} w}  F_{u}(c;y) \cdot \enS_{v}(c';x;y) \\
  &= \sum_{\mu , v} \hat{c}_{\mu,v}^w(y)\,\enS_{w_\mu}(c;x;y)\cdot \enS_{v}(c';x;y).
\end{align*}
Comparing coefficients of $\enS_v$, it follows that $\hat{c}_{\mu,v}^w(y)=0$ unless $\ell(wv^{-1})=\ell(w)-\ell(v)$.  When this length-additivity condition holds, the coefficients arise in the expansion
\[
F_{wv^{-1}}(c;y) =\sum_{\mu} \hat{c}_{\mu,v}^w(y)\,\enS_{w_\mu}(c;x;y).
\]
In the terminology of \cite[\S4]{LLS}, these are the \define{double Edelman-Greene coefficients}, the precise translation being
\[
  \hat{c}^w_{\mu,v}(y) = \hat{c}_{\mu,e}^{wv^{-1}}(y) = j^{wv^{-1}}_\mu(-a)
\]
when $\ell(wv^{-1})=\ell(w)-\ell(v)$ (and $\hat{c}^w_{\mu,v}(y) = 0$ otherwise).

\begin{theorem}\label{t.positive}
The coefficient $\hat{c}^w_{\mu,v}(y)$ lies in $\Z_{\geq 0}[ y_i-y_j \,|\, i\succ j]$.  It is a nonnegative sum of terms which are squarefree in the linear forms $y_i-y_j$, if both indices have the same sign (positive or nonpositive), and have degree at most $2$ in the forms $y_i-y_j$, for $i$ nonpositive and $j$ positive.
\end{theorem}

The total order $\prec$ on $\Z$ is the one defined in Proposition~\ref{p.graham1}, so $1\prec 2 \prec \cdots \prec -1 \prec 0$.  
The theorem refines \cite[Theorem~4.22]{LLS}, which asserts positivity without bounds on the powers of $y_i-y_j$.  The proof given in \cite{LLS} relates the coefficient $\hat{c}_{\mu,v}^w(y)$ to one appearing in the equivariant homology of the affine Grassmannian, and then invokes the quantum-affine (Peterson) isomorphism and positivity in equivariant quantum cohomology.

Our argument is based on a direct application of Graham's positivity theorem \cite{graham}, which says the following.  Suppose $B_N$ is a connected solvable group, with unipotent radical $U_N$ and maximal torus $T$, and $B_0\subset B_N$ is a closed subgroup whose unipotent radical $U_0\subset U_N$ is normalized by $T$.  Let $\chi_1,\ldots,\chi_N$ be the characters of $T$ on the quotient variety $U_N/U_0$ (considered as an affine space), counted with multiplicity.  If $B_N$ acts on a variety $X$, and $Y \subseteq X$ is a $B_0$-invariant subvariety, then there are $B_N$-invariant cycles $Z_I$ so that
\[
  [Y] = \sum_{I\subseteq \{1,\ldots,N\}} \left(\prod_{i\in I} \chi_i \right) [Z_I]
\]
as $T$-equivariant Chow (or homology) classes.  (See also \cite[Ch.~19]{ecag}.)

\begin{proof}
We may compute a given coefficient $c^w_{\mu,v}$ on a sufficiently large but finite dimensional flag variety, so for now we choose $m\gg0$ and set $V=V_{(-m,m]}$, etc., writing $\Fl(V)$ for the complete flag variety, and $\Fl(\bV) = \Fl(m,m+1,\ldots,3m;\bV)$ for the partial flag variety, so the direct sum map is $\boxplus\colon \Gr(m,V)\times \Fl(V)  \to \Fl(\bV)$.  We use the ordered basis $e_{-m+1}, \ldots, e_m$ for $V$, as usual, and let $B^+\subseteq GL(V)$ be the subgroup stabilizing the corresponding flag $V_{\leq \bullet}$.  For $\bV= V\oplus V$, we use the ordered basis
\begin{align*}
&(e_{-m+1},0), \ldots, (e_0,0), (0,e_{-m+1}), \ldots, (0,e_0), \\
& \qquad \qquad \qquad   (0,e_1),\ldots,(0,e_m), (e_1,0), \ldots, (e_m,0).
\end{align*}
The flag $\bV^-_\bullet$ obtained by reading this basis backwards is the one used to define the (opposite) Schubert variety $\bbOmega_w$.  Let $\mathds{B}^-\subseteq GL(\bV)$ be the subgroup stabilizing this flag, and let $\mathds{B}^+$ be the subgroup stabilizing the flag $\bV^+_\bullet$ obtained by reading the basis forwards.

So in our chosen bases for $V$ and $\bV$, the subgroups $B^+\subset GL(V)$ and $\mathds{B}^+\subset GL(\bV)$ are upper-triangular, and $B^-$ and $\mathds{B}^-$ are lower-triangular.  Let $U^+\subset B^+$ and $\mathds{U}^+\subset \mathds{B}^+$ be the corresponding unipotent radicals.


In $\Fl(V)$, the $B^-$ invariant Schubert varieties $\Omega_v$ (of codimension $\ell(v)$) are transverse to $B^+$-invariant Schubert varieties $X_v$ (of dimension $\ell(v)$); likewise one has $\Omega_\mu$ and $X_\mu$ in $\Gr(m,V)$.  The $\mathds{B}^-$-invariant $\bbOmega_w$ and $\mathds{B}^+$-invariant $\mathds{X}_w$ in $\Fl(\bV)$ are defined with respect to the flags $\bV^-_\bullet$ and $\bV^+_\bullet$, respectively.  As we have seen, $\bbOmega_w$ has class $\enS_w(\bbc;x;y)$.

By Poincar\'e duality, we have
\[
  \boxplus_*( [X_\mu \times X_v ] ) = \sum_w \hat{c}^w_{\mu,v}(y) \cdot [\mathds{X}_w]
\]
in $H_T^*\Fl(\bV)$.  The left-hand side is the class of the $(B^+ \times B^+)$-invariant subvariety $\boxplus( X_\mu \times X_v ) \subseteq \Fl(\bV)$.  Applying Graham's theorem expresses this as a sum of $\mathds{B}^+$-invariant cycles, with coefficients coming from the characters of $T$ acting on $\mathds{U}^+/(U^+\times U^+)$.  Since the only $\mathds{B}^+$-invariant cycles are Schubert varieties $\mathds{X}_w$, this is the desired decomposition.

The characters on $\mathds{U}^+/(U^+\times U^+)$ are $y_i-y_j$ for $i\leq 0$ and $j>0$ (each with multiplicity $2$), and $y_i-y_j$ for $i,j\leq 0$ or $i,j>0$ (each with multiplicity $1$).  See Figure~\ref{f.borel} for an illustration.

Finally, by \cite[Theorem~4.22]{LLS}, if $i$ and $j$ have the same sign and $i<j$, the linear forms $y_i-y_j$ do not contribute.
\end{proof}

\begin{figure}
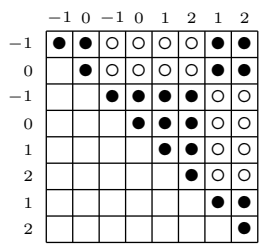

\begin{center}
\pspicture(0,0)(100,100)

\psline(10,10)(90,10)
\psline(10,20)(90,20)
\psline(10,30)(90,30)
\psline(10,40)(90,40)
\psline(10,50)(90,50)
\psline(10,60)(90,60)
\psline(10,70)(90,70)
\psline(10,80)(90,80)
\psline(10,90)(90,90)

\psline(10,10)(10,90)
\psline(20,10)(20,90)
\psline(30,10)(30,90)
\psline(40,10)(40,90)
\psline(50,10)(50,90)
\psline(60,10)(60,90)
\psline(70,10)(70,90)
\psline(80,10)(80,90)
\psline(90,10)(90,90)

\rput[r](5,15){\tiny$2$}
\rput[r](5,25){\tiny$1$}
\rput[r](5,35){\tiny$2$}
\rput[r](5,45){\tiny$1$}
\rput[r](5,55){\tiny$0$}
\rput[r](5,65){\tiny$-1$}
\rput[r](5,75){\tiny$0$}
\rput[r](5,85){\tiny$-1$}

\rput(15,95){\tiny$-1$}
\rput(25,95){\tiny$0$}
\rput(35,95){\tiny$-1$}
\rput(45,95){\tiny$0$}
\rput(55,95){\tiny$1$}
\rput(65,95){\tiny$2$}
\rput(75,95){\tiny$1$}
\rput(85,95){\tiny$2$}

\rput(15,85){$\bullet$}
\rput(25,85){$\bullet$}
\rput(35,85){$\circ$}
\rput(45,85){$\circ$}
\rput(55,85){$\circ$}
\rput(65,85){$\circ$}
\rput(75,85){$\bullet$}
\rput(85,85){$\bullet$}

\rput(25,75){$\bullet$}
\rput(35,75){$\circ$}
\rput(45,75){$\circ$}
\rput(55,75){$\circ$}
\rput(65,75){$\circ$}
\rput(75,75){$\bullet$}
\rput(85,75){$\bullet$}

\rput(35,65){$\bullet$}
\rput(45,65){$\bullet$}
\rput(55,65){$\bullet$}
\rput(65,65){$\bullet$}
\rput(75,65){$\circ$}
\rput(85,65){$\circ$}

\rput(45,55){$\bullet$}
\rput(55,55){$\bullet$}
\rput(65,55){$\bullet$}
\rput(75,55){$\circ$}
\rput(85,55){$\circ$}

\rput(55,45){$\bullet$}
\rput(65,45){$\bullet$}
\rput(75,45){$\circ$}
\rput(85,45){$\circ$}

\rput(65,35){$\bullet$}
\rput(75,35){$\circ$}
\rput(85,35){$\circ$}

\rput(75,25){$\bullet$}
\rput(85,25){$\bullet$}

\rput(85,15){$\bullet$}

\endpspicture
\end{center}

\caption{Weights ($\bullet$)  on $U^+\times U^+$ and ($\circ$) on $\mathds{U}^+/(U^+\times U^+)$  \label{f.borel} }
\end{figure}

\begin{remark}
The proof given in \cite{LLS} relates the coefficient $\hat{c}_{\mu,v}^w(y)$ to one appearing in the equivariant homology of the affine Grassmannian, and then invokes the quantum-affine (Peterson) isomorphism and positivity in equivariant quantum cohomology.  Until the final sentence, our argument is independent of \cite{LLS}.  A completely independent proof, based on a direct transversality argument, appears in \cite{strong-pos}.
\end{remark}

In fact, the direct sum morphism is equivariant for a larger torus.  Let $\mathds{T}=T\times T'$ act on $\bV = V\oplus V$ by characters $y$ on the first factor and $y'$ on the second factor.  Then $\boxplus \colon\sGr(V) \times \sFl(V) \to \sFl(\bV)$ is equivariant for the induced $\mathds{T}$-action.  One can define coefficients $\hat{c}_{\mu,v}^w(y,y') \in \Z[y;y']$ by
\[
 \boxplus^*[\bbOmega_w] = \sum_{\mu,v} \hat{c}_{\mu,v}^w(y,y')\, [\Omega_\mu] \times [\Omega_v],
\]
or equivalently,
\[
 \boxplus_*[X_\mu \times X_v] = \sum_{\mu,v} \hat{c}_{\mu,v}^w(y,y')\, [\mathds{X}_w].
\]
The argument for Theorem~\ref{t.positive} also proves that these coefficients are also Graham-positive:

\begin{theorem}\label{t.positive2}
The coefficient $\hat{c}^w_{\mu,v}(y,y')$ is a nonnegative sum of squarefree monomials in linear forms
\[ 
y_- -y'_-,\; y_- - y'_+,\; y'_- - y_+,\; \text{ and }y'_+-y_+,
\]
where $y_+$ stands for any $y_i$ with $i>0$, $y_-$ for $y_i$ with $i\leq 0$, etc.
\end{theorem}


In other words, the forms appearing are $d-c$ with $c\prec d$, where $c$ and $d$ are among the $y$ and $y'$ variables, ordered so that
\[
  \{y_+\} \prec \{ y'_+\}  \prec \{ y'_- \} \prec \{ y_-\} ,
\]
and exactly one of $c$ or $d$ is a primed variable.  (To compare with the illustration in Figure~\ref{f.borel}, label the rows and columns by $-1,0,-1',0',1',2',1,2$, so that they are scaled by the corresponding characters $y_i$ and $y'_i$.)

The coefficients are equal to the {\em triple Edelman-Greene coefficients} $j^w_\mu(a,b)$ of \cite[\S10]{LLS}, after setting $y_i=-b_i$ and $y'_i=-a_i$; that is, $j^{w}_\mu(a,b) = \hat{c}_{\mu,e}^w(-b,-a)$.  Indeed, the definition shows that  $\hat{c}^w_{\mu,v}(y,y')$ are the coefficients appearing in the expansion
\[
 \enS_w(\bbc;x;y') = \sum_{\mu,v} \hat{c}^w_{\mu,v}(y,y')\, s_\mu(c|y)\,\enS_v(c';x;y'),
\]
which, noting our sign conventions, agrees with the characterization of $j^{wv^{-1}}_\mu(a,b)$ from \cite[\S10]{LLS}.  
So the theorem expresses positivity in the $a$ and $b$ variables, answering a question raised in \cite[Remark~10.13]{LLS}.

One recovers the coefficients $\hat{c}^w_{\mu,v}(y)$ by setting $y'=y$.  However, Theorem~\ref{t.positive} does not follow from Theorem~\ref{t.positive2}, since one can see factors of $y'_i-y_j$ with $i\prec j$.

\begin{example}\label{ex.no-specialize}
We have
\begin{align*}
  \hat{c}^{[2,3,-1,0,1]}_{(2,2),e}(y,y') &= (y'_1-y_2)(y'_1-y_1) 
  \intertext{and}
  \hat{c}^{[2,3,-1,0,1]}_{(1,1),\, [0,2,-1,1]}(y,y') &= y'_1-y_1.
\end{align*}
This shows there is no total order $\prec$ on the variables $(y,y')$ such that both (1) the coefficients $\hat{c}^w_{\mu,v}(y,y')$ are nonnegative sums of monomials in $d-c$, with $c\prec d$, and (2) the specialization $y'=y$ respects the order, i.e., $y_i \prec y'_j$ implies $y_i\prec y_j$.  (Of course, any coefficient $\hat{c}^w_{\mu,v}(y,y')$ violating (2) must map to $0$ under the specialization $y'=y$, as the two shown above do.) 
\end{example}

\begin{remark}
Specializing to the case where $v=w_\nu$ and $w=w_\lambda$, one has coefficients $\hat{c}_{\mu,\nu}^\lambda(y,y')$ for the direct sum morphism of Grassmannians; in particular, they are also positive.  On the other hand, these coefficients do not define a co-commutative coproduct, for the reasons noted in \cite{knutson-lederer}.  
The coefficients displayed in Example~\ref{ex.no-specialize} are $\hat{c}^{(3,3)}_{(2,2),\emptyset}(y,y')$ and $\hat{c}^{(3,3)}_{(1,1),(2,1)}(y,y')$, respectively.  But one computes $\hat{c}^{(3,3)}_{\emptyset,(2,2)}(y,y') = \hat{c}^{(3,3)}_{(2,1),(1,1)}(y,y')=0$.
\end{remark}

Consider the corresponding direct sum morphism $\boxplus\colon \Gr(m,V_{(-m,m]}) \times \Fl(V_{(-m,m]}) \to \Fl(m,m+1,\ldots,3m;\bV_{(-m,m]})$ of finite-dimensional varieties, and identify $\bV_{(-m,m]} = V_{(-2m,2m]}$ using the ordered basis which lists $(e_{i},0)$, and then $(0,e_{i})$.  As before, the image of $X_\mu \times X_v$ under direct sum is a Richardson variety.  Specifically, define a permutation of $\{-2m+1,\ldots,2m\}$ by
\begin{align*}
 \mu \oslash_m v &= [w_\mu(-m+1)-m,\ldots,\,w_\mu(0)-m,\,v(-m+1)+m,\ldots \\ 
 & \qquad \ldots,\, v(m)+m,\,w_\mu(1)-m,\ldots,\,w_\mu(m)-m ].
\end{align*}
For example, for $\mu=(3,1,1)$, $v=[0,-1,2,-2,3,1]$, and $m=3$, we have $\mu \oslash_m v = [-4,-3,0,3,2,5,1,6,4,-5,-2,-1]$.

\begin{proposition}
Assume $m$ is large enough so that $w_\mu$ and $v$ lie in $\Sgp_{(-m,m]}$.  Let $x^{(m)} = [-2m+1,\ldots,\,-m,\,1,\ldots,\,2m,\,-m+1,\ldots,\,0]$.  Then
\[
  \boxplus( X_\mu \times X_v ) = X_{\mu \oslash_m v} \cap \Omega_{x^{(m)}},
\]
a Richardson variety in $\Fl(m,m+1,\ldots,3m;\bV_{(-m,m]})$.
\end{proposition}

The proof is the same as that of \cite[Proposition~2.1]{thomas-yong}.  This leads to another way of computing the Edelman-Greene coefficients.

\begin{corollary}\label{c.coprod-schur-schub}
The polynomial $\hat{c}_{\mu,v}^{w}(y,y')$ is equal to the coefficient of $\enS_{\mu\oslash_m v}(\mathbb{c};x;\mathbb{y})$ in the expansion of $\enS_{w}(\tilde{\mathbb{c}};x;\tilde{\mathbb{y}})\cdot \enS_{x^{(m)}}(\mathbb{c};x;\mathbb{y})$, where
\begin{align*}
\mathbb{y} &= (y_{-m+1},\ldots,y_m,y'_{-m+1},\ldots,y'_m) \intertext{and}
\tilde{\mathbb{y}} &= (y_{-m+1},\ldots,y_0,y'_{-m+1},\ldots,y'_m,y_1,\ldots,y_m),
\end{align*}
and $\mathbb{c}$ and $\tilde{\mathbb{c}}$ are determined by specializing $\prod_{i=-2m+1}^0\frac{1+t_i}{1-x_i}$ to $t=\mathbb{y}$ and $t=\tilde{\mathbb{y}}$, respectively.
\end{corollary}

\begin{proof}
The specializations of the $y$ variables ensure that $\enS_w(\tilde{\mathbb{c}};x;\tilde{\mathbb{y}}) = [\bbOmega_w]$, $\enS_{\mu\oslash_m v}(\mathbb{c};x;\mathbb{y}) = [\Omega_{\mu\oslash_m v}]$, and $\enS_{x^{(m)}}(\mathbb{c};x;\mathbb{y}) = [\Omega_{x^{(m)}}]$ in $H_T^*\Fl(m,m+1,\ldots,3m;\bV_{(-m,m]})$.  And by Poincar\'e duality, the coefficient of $[\Omega_{\mu\oslash_m v}]$ in the expansion of $[\bbOmega_w]\cdot [\Omega_{x^{(m)}}]$ is equal to the (equivariant) integral
\[
 \int_{\Fl(\bV)} [\bbOmega_w]\cdot[\Omega_{x^{(m)}}]\cdot [X_{\mu\oslash_m v}].
\]
We have $[\Omega_{x^{(m)}}]\cdot [X_{\mu\oslash_m v}] = [ \Omega_{x^{(m)}} \cap X_{w_\mu \oslash_m v} ] = [ \boxplus( X_\mu \times X_v )]$, so this integral becomes
\[
 \int_{\Fl(\bV)} [\bbOmega_w]\cdot \boxplus_*[X_\mu \times X_v ] = \int_{\Gr(V)\times\Fl(V)} \boxplus^*[\bbOmega_w] \cdot [X_\mu \times X_v ],
\]
which is the coefficient of $[\Omega_\mu]\times [\Omega_v]$ in $\boxplus^*[\bbOmega_w]$, as claimed.
\end{proof}

\section{Type C}\label{s.typeC}

Most of the foregoing discussion has analogues in other types---in fact, one motivation was to develop a type A analogue of constructions from other classical types.  Here we will discuss some aspects of type C, focusing on the relationship with type A.

Changing notation, we write $T$ for the ``positive'' torus $\prod_{i>0}\C^*$, with standard characters $y_i$ for $i>0$, and $\Tz = T\times \C^*$, where the extra $\C^*$ has character $z$.  This acts on $V$ so that, for $i>0$, $e_i$ has weight $y_i$, and $e_{1-i}$ has weight $z-y_i$.  If we let the larger torus $\left(\prod_{i\in\Z}\C^*\right) \times \C^*$ act on $V$ in the standard way, so that $e_i$ is scaled by $y_i$ for all $i$, then $\Tz$ embeds so that the restriction of characters is $y_i \mapsto y_i$ for $i>0$ and $y_i\mapsto z-y_{1-i}$ for $i\leq 0$.  The corresponding homomorphism of equivariant cohomology rings, $\Z[y][z] \to \Z[y_+][z]$, is defined the same way.

\subsection{Lagrangian Grassmannians and isotropic flag varieties}

We fix a standard symplectic form on $V$, defined by setting
\[
  \langle e_{1-i}, e_i \rangle = -\langle e_i, e_{1-i} \rangle = 1
\]
for $i>0$, and setting all other pairings to $0$.  The form
\[
  \langle \;,\; \rangle \colon V \otimes V \to \C_z
\]
is preserved by $\Tz$, where the target $\C_z$ is scaled by character $z$.  When restricted to each $2m$-dimensional subspace $V_{(-m,m]}$, this defines a symplectic form and an isomorphism
\[
  V_{(-m,m]} \xrightarrow{\sim} V_{(-m,m]}^*\otimes \C_z.
\]
Using these subspaces to define the restricted dual of $V$, this also gives an isomorphism $V \xrightarrow{\sim} V^{*'}\otimes \C_z$.

We fix the flag $V_{\leq\bullet}$ as before.  The \define{infinite Lagrangian Grassmannian} is the subvariety
\[
  \sLG \subseteq \sGr
\]
parametrizing subspaces $E\subseteq V$ which belong to $\sGr$ and are isotropic with respect to the symplectic form, i.e., those $E$ for which $\langle\;,\;\rangle$ becomes identically zero when restricted to $E$.  As for $\sGr$, we use the notation $\sLG(V;V_{\leq\bullet})$ when there is ambiguity in the flag.

The subspace $V_{\leq0}$ is isotropic, so it lies in $\sLG$.  The subspace $V_{>0}$ is also isotropic, but it does not lie in $\sGr$ so does not define a point of $\sLG$.  (Note, however, that the symplectic form defines isomorphisms $V_{\leq0} \isom V_{>0}^{*'} \otimes \C_z$.)

As noted in the introduction, one has compatible embeddings
\[
\begin{tikzcd}
 \LG(m,V_{(-m,m]}) \ar[r,hook] \ar[d,hook] & \LG(m+1,V_{(-m-1,m+1]}) \ar[d,hook] \\
 \Gr(m,V_{(-m,m]}) \ar[r,hook] & \Gr(n+1,V_{(-m-1,m+1]}),
\end{tikzcd}
\]
making $\sLG = \bigcup_{m>0}\LG(m,V_{(-m,m]})$.

The cohomology ring of each finite-dimensional Lagrangian is generated by Chern classes of the tautological bundle $\tS \subseteq V_{(-m,m]}$, with relations coming from the Whitney sum formula.  Using $c=c^T(V_{\leq0}-\tS)$, these relations are determined by $c\cdot \bar{c} = 1$, where
\[
 \bar{c} = c^T(V_{\leq0}^*\otimes \C_z - \tS^*\otimes \C_z).
\]
(Using the symplectic form, one has $V_{(-m,m]}/\tS \isom \tS^*\otimes\C_z$ and $V_{\leq0}^*\otimes \C_z = V_{>0}$, so the relations follow.)  By standard Chern class identities, one writes
\[
 \bar{c}_p = \sum_{i=1}^p \binom{p-1}{i-1}\, (-z)^{p-i} (-1)^i\, c_i,
\]
Extracting the degree $2p$ part of $c\cdot \bar{c}$, one finds relations
\[
 C_{pp} := \sum_{0\leq i\leq j \leq p} (-1)^j \left( \binom{j}{i}+\binom{j-1}{i} \right) z^i\, c_{p-i+j}\, c_{p-j} = 0, 
\]
for $p>0$.  Taking the limit, we have
\[
  H_{\Tz}^*\sLG = \bGamma[y_+],
\]
where
\[
  \bGamma = \Lambda[z]/( C_{pp} )_{p>0}.
\]
Pullback by the inclusion $\sLG \hookrightarrow \sGr$ induces the canonical surjection $\Lambda[z][y] \tto \bGamma[y_+]$.

For $k\leq 0$, one defines $\sIG^k \subseteq \sGr^k$ in the same way.  It is the union
\[
  \sIG^k = \bigcup_{m>|k|} \IG(m+k,V_{(-m,m]})
\]
of (possibly non-maximal) isotropic Grassmannians.  The (type C) \define{infinite isotropic flag variety} is the variety
\[
  \sIFl = \{ E_\bullet : (\cdots \subset E_{-1} \subset E_{0} =E \subset V) \,|\, E_i \in \sIG^{i} \},
\]
a subvariety of $\prod_{k\leq 0} \sIG^k$.  Its cohomology ring is
\[
  H_{\Tz}^*\sIFl = \bGamma[x_+,y_+],
\]
using $x_i = c_1^{\Tz}(\tS_{-i+1}/\tS_{-i})$ for $i>0$, where $(\cdots \subset S_{-1} \subset S_{0} =S \subset V)$ is the tautological flag.  (As usual, these should be regarded as the stable limits of vector bundles on the finite-dimensional type C flag varieties.)

Just as for finite-dimensional varieties, an isotropic flag extends canonically to a complete flag, by $E_{i} = E_{-i}^\perp$ for $i>0$, and one obtains an embedding $\sIFl \hookrightarrow \sFl$.  Using the symplectic form to identify $V\isom V^{*'}\otimes\C_z$, this realizes $\sIFl$ as the fixed locus for the duality involution described in  \S\ref{ss.duality} (or rather, a variation of that involution which twists by $\C_z$, see \cite{part1}).  
In particular, we have $E_{i}/E_{i-1} \isom (E_{1-i}/E_{-i})^*\otimes \C_z$ for $i\geq1$.

The pullback on cohomology is the surjection $\Lambda[z][x,y] \tto \bGamma[x_+,y_+]$, where  $x_i \mapsto x_i$ for  $i>0$, and $x_i\mapsto z-x_{1-i}$ for $i\leq0$.  Realizing $\sIFl\subset \sFl$ is the fixed locus of a (twisted) duality involution gives another way of viewing the relations defining this quotient of $\Lambda[z][x,y]$.  The corresponding homomorphism
\[
  \bm\omega(c_k) = \sum_{i=1}^k \binom{k-1}{i-1} (-z)^{k-i} S_{1^i}(c), \quad \bm\omega(x_i)=z-x_{-i}, \quad \bm\omega(y_i) = z-y_{-i}
\]
must be the identity on $H_\Tz^*\sIFl$, and the relations express this.

\begin{remark}\label{r.c-untwisted}
The ring $\Gamma = \bGamma/(z)$ is the classical ring of Schur $Q$-polynomials.  This can be written as $\Gamma=\Lambda/(C_{pp})_{p>0}$, where now $C_{pp} = \sum_{j=0}^p (-1)^j c_{p+j}\,c_{p-j}$.  Many statements and formulas become much simpler in the ``untwisted'' case where $z=0$.
\end{remark}

\begin{remark}\label{r.gamma-lambda}
In symmetric function theory, one often embeds $\Gamma \hookrightarrow \Lambda$, considering both as rings of symmetric functions in auxiliary variables.  The ring $\bGamma$ also embeds in $\Lambda[z]$.  This requires more care, but it also points the way to a geometric interpretation.  It is helpful to realize these inclusions of rings as pullbacks via a different map between infinite Grassmannians.  We will describe it in terms of compatible maps of finite-dimensional varieties.

To lighten the notation, let $V_m = V_{(-m,m]}$ and $L=\C_z$, and let $\bV_m = V_m \oplus V^{*}_m \otimes L$, with its canonical $L$-valued symplectic form.  For any fixed $k$, there is a map
\[
   \Gr(m+k,V_m) \hookrightarrow \LG(\bV_m),
\]
sending a point $A\subset V_m \twoheadrightarrow B$ to $A\oplus B^*\otimes L \subset \bV$.  One checks that this is an isotropic subspace.  
The space $\mathds{E}_m = V_{\leq 0} \oplus V_{>0}^* \otimes L \subset \bV_m$ is also isotropic subspace.  Let $\tbS\subset\bV_m$ be the tautological bundle.  Pullback sends $c^{\Tz}(\bV_m - \tbS - \mathds{E}_m)$ to
\[
 c^\Tz(\bV_m - S - Q^*\otimes L - \mathds{E}_m ) = c^\Tz( V_{>0}-V_{>0}^*\otimes L + S^*\otimes L - S),
\]
where $S \subset V_m \twoheadrightarrow Q$ are tautological bundles on $\Gr(m+k,V_m)$.

These maps are all compatible with the natural inclusions $V_m\subset V_{m+1}$.  So there is a corresponding morphism $\sGr^{(k)}(V) \to \sLG(\bV)$.  The corresponding pullback map on cohomology, $\bGamma \to \Lambda[y_+][z]$ is given by
\begin{equation}\label{e.gamma-embed}
  c \mapsto \prod_{i> 0} \frac{1+y_i}{1-y_i+z} \prod_{i\leq k} \frac{1+x_i+z}{1-x_i},
\end{equation}
where $x_{-m+1},\ldots,x_k$ are Chern roots of $S^*$ on each finite-dimensional $\Gr(m+k,V_m)$, and $\Lambda$ is regarded as the ring of supersymmetric functions in the variables $x_i$ for $i\leq k$ and $y_i$ for $i>0$.  The series on the right-hand side of \eqref{e.gamma-embed} is stable with respect to setting $x_i=y_i=0$ for $|i|>m$, so its homogeneous pieces are well-defined elements of $\Lambda[y_+][z]$, as they must be.  (They are deformations of the classical polynomials $Q_p(x)$.)
\end{remark}

\subsection{Schubert varieties and Schubert polynomials}

The group of \define{signed permutations} is the subgroup $W_\infty \subset \Sgp_{\Z}$ of permutations $w$ such that $w(1-i)=1-w(i)$ for all $i$.  These are the elements of $\Sgp_\Z$ which are fixed by the involution $\omega$.  The submonoid $\mathrm{SgnInj}(\Z) \subset \Inj(\Z)$ is defined similarly, and one also has the submonoid $\mathrm{SgnInj}^0(\Z) \subset \mathrm{SgnInj}(\Z)$ of signed injections with finitely many sign changes.  (The balancing condition is automatic here.)  Choosing a large enough $m$ so that $w(i)=i$ for $|i|>m$, we often write $w\in W_\infty$ in \define{one-line notation} as $w=[w(1),\ldots,w(m)]$.

Just as $\Inj^0(\Z)$ indexes fixed points of $\sFl$, the subset $\mathrm{SgnInj}^0(\Z)$ indexes fixed points of $\sIFl$: the point $p_w$ corresponds to the flag $E_\bullet$ with $E_k$ spanned by $e_{w(i)}$ for $i\leq 0$.  (With conventions as in \S\ref{s.fp} for integers not in the image of $w$.)

Schubert varieties are indexed by signed permutations.  For each $w \in W_\infty$, there is a Schubert variety in $\sIFl$, defined by
\[
  \Omega_w = \{ E_\bullet \,|\, \dim(E_p \cap V_{>q}) \geq k_w(p,q) \text{ for } p\leq 0 \text{ and all }q \},
\]
where $k_w(p,q) = \#\{ a\leq p \, |\, w(a) >q \}$, as before.

A \define{strict partition} $\lambda = (\lambda_1>\cdots>\lambda_s>0)$ determines a \define{Grassmannian signed permutation} $w=w_\lambda$ by setting $w(i)=1-\lambda_i$ for $1\leq i\leq s$, and filling in the remaining unused values in increasing order.  For example, $\lambda = (4,2,1)$ has Grassmannian signed permutation $w_\lambda=[-3,-1,0,3]$.  
Schubert varieties $\Omega_\lambda \subseteq \sLG$ are defined by conditions $\dim( E \cap V_{>\lambda_k} )\geq k$.

As before, Schubert varieties in $\sIFl$ determine unique Schubert classes.  The \define{(twisted) double Schubert polynomial} of type C is the polynomial such that
\[
  \tenS_w^C(c;x;y) = [\Omega_w]
\]
under $\bGamma[x_+,y_+] = H_{\Tz}^*\sIFl$.  For $z=y=0$, this is precisely the definition in \cite{BH}; for $z=0$, these are the double Schubert polynomials of \cite{IMN}.  Among the many wonderful properties of these polynomials, we mention the Cauchy formula:
\begin{align}\label{e.cauchyC}
  \tenS_w^C(\bbc;x;y) &= \sum_{uv \dot{=} w} \tenS_v^C(c;x;t) \, \tenS_u^C(c';z-t;y),
\end{align}
where $\bbc=c\cdot c'$.  

One can compare Schubert polynomials in types A and C via the canonical surjection $\Lambda[z][x,y] \to \bGamma[x_+,y_+]$: for $w\in \Sgp_+\subset W_\infty$, this map sends $\tenS_w^A(c;x;y)$ to $\tenS_w^C(c;x;y)$.  A geometric proof is in \cite{part1}.

The \define{twisted double $Q$-polynomials} $\bm{Q}_\lambda(c|y) = \tenS_{w_\lambda}^C(c;x;y)$ correspond to Schubert classes in $\sLG$, so they form a basis for $\bGamma[y_+]$ over $\Z[z][y_+]$.  At $z=0$ (and an appropriate evaluation of $c$), these specialize to Ivanov's double $Q$-functions; at $z=y=0$, they specialize to Schur's $Q$-polynomials $Q_\lambda(c)$, which form a basis for $\Gamma$.

\subsection{Direct sum and coproduct}

The embedding $\sLG \subset \sGr$ is compatible with the direct sum map, where one takes the symplectic form on $\bV=V\oplus V$ to be the difference of symplectic forms on each summand.  So one obtains a coproduct $\Delta\colon \bGamma[y_+] \to \bGamma[y_+]\otimes_{\Z[y]} \bGamma[y_+]$.  Similarly, the direct sum morphism $\sLG(V) \times \sIFl(V) \to \sIFl(\bV)$ determines a co-module homomorphism $\bGamma[x_+;y_+] \to \bGamma[y_+]\otimes_{\Z[y]} \bGamma[x_+,y_+]$.

In Schubert classes, we can again write
\[
  \boxplus^*[\Omega_w] = \sum_{\mu,v} \hat{f}_{\mu,v}^w(y;z) [\Omega_u] \times [\Omega_v],
\]
for strict partitions $\mu$ and signed permutations $v,w$, where the polynomials $\hat{f}_{\mu,v}^w(y;z)$ are \define{type C double Edelman-Greene coefficients}. 

Using Cauchy formulas, this co-module operation on Schubert polynomials can be written as
\begin{align*}
 \tenS_w^C(\bbc;x;y)  &= \sum_{uv\dot{=}w}  \bm{F}^C_u(c;y) \cdot \tenS_v^C(c';x;y) \\
 &= \sum_{\mu,v} \hat{f}_{\mu,v}^w(y;z) \, \bm{Q}_\mu(c|{y}) \, \tenS_v^C(c';x;y) ,
\end{align*}
where the \define{(twisted) double type C Stanley polynomial} is defined as
\[
  \bm{F}^C_w(c;y) = \tenS^C_w(c;z-y;y).
\]
As before, the coefficients $\hat{f}_{\mu,v}^w(y;z)$ arise in the expansion of $\bm{F}^C_{wv^{-1}}$ in the $\bm{Q}_\mu$ basis.

Also as before, the direct sum morphism is actually equivariant with respect to the larger $T \times T\times (\C^*)$ action on $\bV=V\oplus V$, where the $\C^*$ factor still acts diagonally (though once again, the extended equivariant structure does not define a commutative coproduct).  Writing $y_i$ for the characters on the first factor and $y'_i$ for those on the second factor, we can expand
\[
  \boxplus_*[X_\mu  \times  X_v] = \sum_{\mu,v}\hat{f}_{\mu,v}^w(y,y';z) \, [\mathds{X}_w]
\]
in $H_{T \times T\times (\C^*)}^*\sIFl(\bV)$.

\begin{theorem}\label{t.positiveC}
The coefficient $\hat{f}^w_{\mu,v}(y,y';z)$ is a nonnegative sum of squarefree monomials in linear forms $-y'_i-y_j+z$ and $y'_i-y_j$.
\end{theorem}

The proof is the same as for Theorems~\ref{t.positive} and \ref{t.positive2}, applying Graham's theorem and keeping track of weights on the corresponding unipotent groups in symplectic groups.  Specializing $y=y'$, one obtains the type C analogue of a weak form of Theorem~\ref{t.positive}.

\begin{remark}
In the Lagrangian Grassmannian case where $w=w_\lambda$ and $v=w_\nu$ for strict partitions $\lambda$ and $\nu$, the polynomial $\hat{f}_{\mu,\nu}^\lambda(y)$ may be regarded as a \define{dual Hall-Littlewood coefficient}.  It expresses the coproduct
\[
  \bm{Q}_\lambda(\bbc|y) = \sum_{\mu,\nu} \hat{f}_{\mu,\nu}^\lambda(y;z)\, \bm{Q}_\mu(c|y) \cdot \bm{Q}_\nu(c'|y),
\]
where $\bbc=c\cdot c'$ as usual. 
Evaluating at $y=z=0$, this is the structure constant for multiplication in the basis of $P$-Schur functions; that is, $\hat{f}_{\mu,\nu}^\lambda(0) = f_{\mu,\nu}^\lambda$ in the notation of \cite[\S III.5]{Mac1}.  Combinatorial formulas for this case were given by Stembridge \cite{stembridge}.  
\end{remark}



\noindent
{\sc Department of Mathematics, The Ohio State University, Columbus, OH 43210}

\noindent
{\it Email address:} anderson.2804@math.osu.edu

\end{document}